\documentclass[
final
]{dmtcs-episciences}

\usepackage[utf8]{inputenc}
\usepackage{subfigure}

\author{Alexander Burstein\affiliationmark{1}
  \and Opel Jones\affiliationmark{2} 
  }

\title{Enumeration of Dumont permutations avoiding certain four-letter patterns}

\affiliation{
  Department of Mathematics, Howard University, Washington, DC 20059, USA\\
  Department of Mathematics, Towson University, Towson, MD 21252, USA
  }


\received{2020-02-28}
\revised{2021-05-07}
\accepted{2021-05-12}
\publicationdetails{22}{2021}{2}{7}{6174}

\usepackage{amsthm,amsfonts,amssymb,amsopn,
amsmath,indentfirst,tikz,hyperref,varioref}
\usepackage{extpfeil}

\newtheorem{theorem}{Theorem}[section]

\newtheorem{lemma}[theorem]{Lemma}

\newtheorem{conjecture}[theorem]{Conjecture} 

{
\theoremstyle{definition}
\newtheorem{definition}[theorem]{Definition}
\newtheorem{notation}[theorem]{Notation}
\newtheorem{example}[theorem]{Example}
}

{
\theoremstyle{remark}

}

\begin{document}

\maketitle


\begin{abstract}
In this paper, we enumerate Dumont permutations of the fourth kind avoiding or containing once certain permutations of length 4. We also conjecture a Wilf-equivalence of two 4-letter patterns on Dumont permutations of the first kind.
\end{abstract}

\section{Preliminaries}

This paper is concerned with the enumeration of pattern-restricted Dumont permutations. We will begin by defining patterns and Dumont permutations of various kinds, the original two kinds defined by Dumont \cite{Dumont} and the more recent third and fourth kinds defined by Burstein et al. \cite{Burstein_New}. We will then discuss previous results on this topic and present several new avoidance results for Dumont permutations of the fourth kind, an enumeration result for Dumont permutations of the fourth kind containing a single occurrence of a pattern, as well as a Wilf-equivalence conjecture on restricted Dumont permutations of the first kind.

\subsection {Patterns}


We begin with an example of pattern containment.  Suppose we are given a permutation, say $p=26483751$, and another permutation usually of shorter or equal length, e.g. $q=132$.  We say that the string $265$ is \emph{order-isomorphic} to $132$.  In such a case, we say that $p$ contains $q$ as a \emph{pattern}, and that $265$ is an \emph{instance}, or \emph{occurrence}, of $q$ in $p$.  Notice also that $265$ is not the only instance of $q$ in $p$, so are the subsequences $287$ and $475$. 

On the other hand, for the same $p$ and a pattern $q'=1234$,  we see that there is no length $4$ subsequence of $p$ that is order-isomorphic to $q'$. In this case we say that $p$ \emph{avoids} $q'$.  Let us now define pattern containment and pattern avoidance more formally.

\subsection{Pattern avoidance}

Simion and Schmidt \cite{Simion} was the first paper devoted solely to pattern avoidance. Since then, pattern avoidance has been the subject of numerous papers. We will only mention here a few introductory sources:
\begin{itemize}

\item Bevan \cite{Bevan} contains essential terminology used in many patterns on permutations patterns,

\item  Chapter 4 of B\'ona's \emph{Combinatorics of Permutations} \cite{Bona} contains a textbook introduction to the subject, 

\item Vatter's ``Permutation Classes'' \cite{Vatter} in \emph{Handbook of Enumerative Combinatorics} is an in-depth survey of the literature and foundational results. 

\end{itemize}



\begin{definition}[\cite{Bona}]
\label{def:avoidance}
Let $S_n$ be the set of permutations of $[n] = \{1,2,3,...,n\}$. Let $p=(p_1,p_2,...,p_n)\in S_n$ and $q=(q_1,q_2,...,q_k)\in S_k$ be two permutations.  Then we say that $p$ \emph{contains} $q$ if there is a $k$-tuple $(p_{i_1},p_{i_2},\dots,p_{i_k})$ in $p$ such that $1\le i_1<i_2<\dots<i_k\le n$, and $p_{i_\alpha}<p_{i_\beta}$ if and only if $q_\alpha<q_\beta$.  Otherwise, we say that $p$ \emph{avoids} $q$.  
\end{definition}

\begin{notation}
For a set of permutations avoiding a pattern or set of patterns, we let $S_n(\tau)$ be the set of permutations in $S_n$ avoiding a pattern $\tau$, and we let $S_n(T)$ be the set of permutations avoiding a set of patterns $T$, in other words, simultaneously avoiding all patterns in $T$.
\end{notation}


\begin{notation}
Let $\tau'$ and $\tau''$ be subsequences of a permutation $\tau$.  We say that $\tau' > \tau''$ if every entry of $\tau'$ is greater than every entry of $\tau''$.  We define $\tau' < \tau''$ similarly.
\end{notation}

\begin{notation}
For a string $\tau$ and an integer $m$, the string $\tau + m$ (resp. $\tau - m$) is obtained by adding $m$ to (resp. subtracting $m$ from) every entry of $\tau$.
\end{notation}

Thus, for example, if $\tau$ is a permutation of $[n]$, then $\tau+m$ is a permutation of $[m+1,m+n]$.

Since $|S_n(12)| = |S_n(21)| = 1$ for all $n \ge 0$, we are interested in patterns of length $n \ge 3$. For permutations avoiding patterns of length $n=3$, Knuth \cite{Knuth} and Simion and Schmidt \cite {Simion} showed that permutations avoiding any $\tau \in S_3$ are Wilf-equivalent, i.e. have the same enumeration sequence.  More precisely, for all $\tau\in S_3$ and $n \ge 0$, $|S_n(\tau)| = C_n$, where $C_n$ is the $n$-th Catalan number. 


\subsection{Permutation symmetries and Wilf-equivalence}


Permutation symmetries are quite convenient, in that they help us find equinumerously avoided patterns (or sets of patterns, respectively). 

\begin{definition}
Let patterns $\tau_1$ and $\tau_2$ be such that $|S_n(\tau_1)| = |S_n(\tau_2)|$ for any $n \ge 0$.  Also, let sets of patterns $T_1$ and $T_2$ be such that $|S_n(T_1)| = |S_n(T_2)|$ for any $n \ge 0$.  Then such patterns (or sets of patterns, respectively) are called \emph{Wilf-equivalent} and said to belong to the same \emph{Wilf class}.
\end{definition}

\begin{definition}
Let $\pi \in S_{n}$ be a permutation, and let $\pi(i)$ be the $i$-th entry of $\pi$ from left to right.  Define a \emph{permutation diagram of $\pi$} as the set of dots $\{(i, \pi(i))\, |\, 1 \le i \le n\}$.
\end{definition}

Visually, we can represent $\pi$ by placing dots on an $n \times n$ board in the following fashion: place a dot inside the square in the $i$-th column from the left and $\pi(i)$-th row from the bottom.  Thus, we read $\pi$ from left-to-right and bottom-to-top, as mentioned earlier.  Note, the origin of our $n \times n$ board is at the bottom-left corner.  

%
%
%
%
%
%
%

Now let us define a few symmetry operations on $S_n$ that map every pattern onto a Wilf-equivalent pattern on the ambient set $S_n$.

\begin{definition}
Let $\pi \in S_n$ be a permutation, read left-to-right and bottom-to-top.

\begin{itemize}
\item The \emph{reverse} of $\pi$, denoted $\pi^\text{\hspace{0.01in}r}$, is $\pi$ read right-to-left and bottom-to-top.
\item The \emph{complement} of $\pi$, denoted $\pi^\text{\hspace{0.01in}c}$, is $\pi$ read left-to-right but top-to-bottom.
\item The \emph{composition of the reverse and the complement} of $\pi$, denoted  $\pi^\text{\hspace{0.01in}r\hspace{0.01in}c}$, is $\pi$ read right-to-left and top-to-bottom. Note that $\pi^\text{\hspace{0.01in}r\hspace{0.01in}c} = \pi^\text{\hspace{0.01in}c\hspace{0.01in}r}$.
\item The \emph{inverse} of $\pi$, denoted $\pi^{-1}$, is the usual inverse of a permutation.  That is, the position of each entry and the entry itself of $\pi$ are switched in $\pi^{-1}$.
\end{itemize}

\end{definition}

%

\begin{definition}

The symmetries above generate the group of symmetries of the square and yield the set of patterns
$\{ 
\pi, 
\pi^\text{\hspace{0.01in}r}, 
\pi^\text{\hspace{0.01in}c}, 
\pi^\text{\hspace{0.01in}r\hspace{0.01in}c}, 
\pi^{-1},
(\pi^{-1})^\text{\hspace{0.01in}r}, 
(\pi^{-1})^\text{\hspace{0.01in}c}, 
(\pi^{-1})^\text{\hspace{0.01in}r\hspace{0.01in}c}
\}$,
called the \emph{symmetry class} of $\pi$.

\end{definition}

\begin{example}

%

Let $\pi = 263541$.  Then the symmetry class of $\pi$ is the set of patterns
\[
\{263541,\ 145362,\ 514236,\ 632415,\ 613542,\ 245316,\ 164235,\ 532461\},
\]
\end{example}

Note the following property of the symmetry classes of permutations: if $\pi$ avoids a pattern $\rho$, then each symmetry of $\pi$ avoids the same symmetry of $\rho$.  In other words, if $\pi$ avoids $\rho$, then $\pi^\text{\hspace{0.01in}r}$ avoids $\rho^\text{\hspace{0.01in}r}$,
$\pi^\text{\hspace{0.01in}c}$ avoids $\rho^\text{\hspace{0.01in}c}$, and so on.
Thus, each element of the symmetry class of $\pi$ is Wilf-equivalent to $\pi$.

\begin{definition}
A \emph{fixed point} (resp. \emph{excedance}, \emph{deficiency}) of a permutation $\pi$ is a position $i$ such that $\pi(i) =  i$ (resp. $\pi(i) > i$, $\pi(i) <  i$).
\end{definition}

\section {Dumont permutations}

Dumont permutations are a certain class of permutations shown by Dumont \cite{Dumont} to be counted by the (unsigned) Genocchi numbers \cite[\href{https://oeis.org/A110501}{A110501}]{OEIS}, a multiple of the Bernoulli numbers.  Let the Genocchi and Bernoulli numbers be denoted $G_{2n}$ and $B_{2n}$, respectively, then we have $G_{2n} = 2 (1 - 2^{2n})(-1)^n B_{2n}$.

The exponential generating functions for the unsigned and signed Genocchi numbers are given by
$$\sum_{n=1}^\infty G_{2n} \frac {x^{2n}}{(2n)!} = x \tan \frac {x}{2} \qquad \text{and} \qquad \sum_{n=1}^\infty (-1)^n G_{2n} \frac {x^{2n}}{(2n)!} = \frac {2x}{e^x + 1} - x = - x \tanh \frac {x}{2}.$$

%

The definition of Dumont permutations of the first and third kinds involves descents and ascents, and thus, the linear structure of permutations. The definition of Dumont permutations of the second and fourth kinds involves fixed points, excedances, and deficiencies, and thus the cyclic structure of permutations. 

\begin{definition}
We say that $i\in[n-1]$ is a \emph{descent} of $\pi$ if $\pi(i) > \pi(i+1)$.  
In this case, we call $\pi(i)$ the \emph{descent top} and $\pi(i+1)$ the \emph{descent bottom}. We also say that $i\in[n-1]$ is an \emph{ascent} of $\pi$ if $\pi(i) < \pi(i+1)$. In this case, we call $\pi(i)$ the \emph{ascent bottom} and $\pi(i+1)$ the \emph{ascent top}.
\end{definition}

\begin{definition}
A \emph{fixed point} (resp. \emph{excedance}, \emph{deficiency}) of a permutation $\pi$ is a position $i$ such that $\pi(i) =  i$ (resp. $\pi(i) > i$, $\pi(i) <  i$).
\end{definition}

We only consider Dumont permutations of even size since it will be clear from the definitions to follow that Dumont permutations of each kind of size $2n+1$ are obtained by appending the fixed point $2n+1$ to the right of a Dumont permutation of the same kind of size $2n$.

\subsection{Dumont permutations of the first and second kinds}

\begin{definition}
A \emph{Dumont permutation of the first kind} (or \emph{Dumont-1 permutation} for short) is a permutation wherein each even entry must be immediately followed by a smaller entry, and each odd entry must be immediately followed by a larger entry, or ends the permutation (i.e. the last entry must be odd).  In other words, for all $i$, $1 \leq i \leq 2n$, and some $k$, $1 \leq k < n$,
\begin{align*}
\pi(i) = 2k &\implies i < 2n \quad \text{and} \quad \pi(i) > \pi(i+1), 
\\
\pi(i) = 2k-1 &\implies \pi(i) < \pi (i+1) \quad \text{or} \quad i = 2n.
\end{align*}
\end{definition}

\begin{example}
\label {Dumont Example 1}
Let $\pi = 435621 \in S_6$. Then $2$, $4$, and $6$, the descent tops of $\pi$, are all even, while the descent bottoms of $\pi$ are $3$ and $5$, and its final entry is $1$, all odd.  Thus, $\pi$ is a Dumont permutation of the first kind.
\end{example}

\begin{definition}
A \emph{Dumont permutation of the second kind } (or \emph{Dumont-2 permutation} for short) is a permutation wherein each entry at an even position is a deficiency, and each entry at an odd position is a fixed point or an excedance.  In other words, for all $i$, $1 \leq i \leq n$,
\[
\pi(2i) < 2i, \qquad
\pi(2i-1) \geq 2i-1.
\]
\end{definition}

\begin{example}
\label {Dumont Example 2}
Let $\rho = 614352 \in S_6$. We see that the deficiencies of $\rho$ occur at positions $2$, $4$ and $6$, all even while excedances occur at positions $1$ and $3$, and a fixed point is at position $5$, all odd. Thus, $\rho $ is a Dumont permutation of the second kind.
\end{example}

\begin{notation} The Dumont permutations of the first (resp. second) kind shall be denoted by $\mathfrak{D}^1_{2n}$ (resp. $\mathfrak{D}^2_{2n}$).
\end{notation}

Thus, from Examples \ref{Dumont Example 1}, and \ref{Dumont Example 2}, we have $\pi \in \mathfrak{D}^1_{6}$, and $\rho \in \mathfrak{D}^2_{8}$.  As stated earlier, Dumont \cite {Dumont} showed that $|\mathfrak{D}^1_{2n}| = |\mathfrak{D}^2_{2n}| = G_{2n+2}$.  There is a simple natural bijection \cite {Dumont_Foata} between Dumont permutations of the first and second kinds, namely, the \emph{Foata's fundamental transformation} \cite{Foata} that maps left-to-right maxima to cycle maxima.


\begin{definition}
An entry $\pi(i)$ of a permutation $\pi$ is a \emph{left-to-right maximum} if $\pi(i)>\pi(j)$ for all $j<i$.
\end{definition}

In particular, the first entry of $\pi$ is always a left-to-right maximum $\pi$.

The map $f$, called the \emph{Foata's fundamental transformation} \cite{Foata}, is defined as follows:

\begin{itemize}
\item Start with a Dumont permutation of the first kind, say $\pi$.
\item Insert parentheses to display the permutation in cycle notation, so that each cycle starts with a left-to-right maximum.
\end{itemize}

Then $f$ is a bijection, and $\pi\in\mathfrak{D}^1_{2n}$ if and only if $f(\pi)\in\mathfrak{D}^2_{2n}$.

\begin{example}
For $\pi = 435621 \in \mathfrak{D}^1_{6}$, we have the left-to-right maxima $\pi(1)=4$, $\pi(3)=5$, and $\pi(4)=6$, so
$$f(\pi)= (43) (5) (621) \, = \, 614352.$$
Notice that $f(\pi) = 614352 \in \mathfrak{D}^2_{6}.$
\end{example}

%

\subsection{Dumont permutations of the third and fourth kinds}

With respect to the next two types of Dumont permutations, Kitaev and Remmel \cite {Kitaev_Remmel_1, Kitaev_Remmel_2} first conjectured that sets of permutations where each descent is from an even value to an even value are also counted by the Genocchi numbers.  Burstein and Stromquist \cite {Burstein_Third} proved this conjecture and called those sets of permutations Dumont permutations of the third kind. As with Dumont-1 and Dumont-2 permutations, the Foata's fundamental transformation maps Dumont permutations of the third kind onto a related equinumerous set that Burstein and Stromquist  \cite {Burstein_Third} called Dumont permutations of the fourth kind.  We now introduce these two sets.

\begin{definition}
A \emph{Dumont permutation of the third kind} (or \emph{Dumont-3 permutation} for short) of size $2n$ ($n\in\mathbb{N}$) is a permutation where each descent is from an even value to an even value, that is all descent tops and all descent bottoms are even.  In other words, for all $i$, $1 \leq i \leq 2n-1$,
$$\pi(i) > \pi(i+1) \implies \pi (i) = 2l \, \text{ and }\, \pi(i+1) = 2k$$
for some $k$ and $l$, $1 \leq k < l \leq n$.
\end{definition}

\begin{example}
\label {Dumont Example 3}
Let $\pi = 16238457 \in S_8$. Then the descent tops of $\pi$ are $6$ and $8$ are and the descent bottoms are $2$ and $4$, all even entries. Thus, $\pi$ is a Dumont permutation of the third kind.
\end{example}

\begin{definition}
A \emph{Dumont permutation of the fourth kind} (or \emph{Dumont-4 permutation} for short) of size $2n$ ($n\in\mathbb{N}$) is a permutation where deficiencies must be even values at even positions. In other words, for all $i$, $1 \leq i \leq 2n$,
$$\pi(i) < i \implies i = 2u \, \text{ and } \, \pi(i) = 2v$$
for some $1 \leq v < u \leq n$.
\end{definition}


\begin{example}
\label {Dumont Example 4}
Let $\rho = 13657284 \in S_8$. Then the deficiencies of $\rho$ are $\pi(6)=2$, $\pi(8)=4$, where all the positions and entries involved are even.  Thus, $\rho$ is a Dumont permutation of the fourth kind.
\end{example}

\begin{notation} Following the same notation as earlier, the Dumont permutations of the third (resp. fourth) kind shall be denoted by $\mathfrak{D}^3_{2n}$ (resp. $\mathfrak{D}^4_{2n}$).
\end{notation}

Thus, from Examples \ref {Dumont Example 3}, and \ref {Dumont Example 4}, we have $\pi \in \mathfrak{D}^3_{8}$, and $\rho \in \mathfrak{D}^4_{8}$.  The same bijection, Foata's fundamental transformation, yields $|\mathfrak{D}^3_{2n}| = |\mathfrak{D}^4_{2n}| = G_{2n+2}$. (In fact, notice that $f(16238457)=13657284$ in Examples \ref{Dumont Example 3} and \ref{Dumont Example 4} above.)

We now reference the proof \cite {Burstein_Third} that  $|\mathfrak{D}^1_{2n}| = |\mathfrak{D}^3_{2n}| = G_{2n+2}$, thus showing that all four kinds of Dumont permutations are counted by the Genocchi numbers.

\subsection{Pattern avoidance in Dumont-1 and Dumont-2 permutations}

There have been several enumerations of pattern-restricted Dumont-1 and Dumont-2 permutations, mostly by Mansour, Burstein, Elizalde, and Ofodile \cite{Mansour, Burstein_Restricted, Burstein_Dyck, Ofodile}.  We will first reproduce their theorems for enumerations of Dumont-1 permutations avoiding patterns of length three.

\subsubsection{Patterns of length three}

For the six patterns of length three, all the corresponding pattern-restricted sets have been enumerated, both in Dumont-1 and Dumont-2 permutations.

\begin{theorem}[\cite{Mansour}] 
For all $n\ge 0$,
\[
|\mathfrak{D}^1_{2n} (132)| = |\mathfrak{D}^1_{2n} (231)| = |\mathfrak{D}^1_{2n} (312)| = C_n = \frac{1}{n+1} \binom{2n}{n}.
\]
\end{theorem}

\begin{theorem}[\cite{Burstein_Restricted}] 
For all $n\ge 1$,
\[
|\mathfrak{D}^1_{2n} (213)| = 
C_{n-1}.
\]
\end{theorem}

It is important to note here that although reverses, complements, and inverses of patterns create symmetry classes and Wilf-equivalences in the set of all permutations, Dumont permutations are not closed under any symmetry operations, so that applying a symmetry operation to a pattern does not necessarily yield a pattern-avoiding set of the same cardinality. For example, patterns $132$ and $312$ are complements of each other, and $|\mathfrak{D}^1_{2n} (132)| = |\mathfrak{D}^1_{2n} (312)| = C_n$; on the other hand, patterns $213$ and $231$ are also complements of each other, but $|\mathfrak{D}^1_{2n} (213)| = C_{n-1} \ne C_n =  |\mathfrak{D}^1_{2n} (231)|$. 

\begin{theorem}[\cite{Burstein_Restricted}] 
For all $n\ge 0$,
\[
|\mathfrak{D}^1_{2n} (321)| = 1, \quad \text{namely,} \quad \mathfrak{D}^1_{2n} (321) = \{2, 1, 4, 3, \dots, 2n, 2n-1\}.
\]
\end{theorem}

\begin{theorem}[\cite{Burstein_Restricted}] 
For all $n\ge 3$,
\[
|\mathfrak{D}^1_{2n} (123)| = 4, \quad \text{namely,} \quad \mathfrak{D}^1_{2n} (123) = \{(2n-1, 2n, 2n-3, 2n-2, \dots, 7, 8, \pi)\, |\, \pi \in \mathfrak{D}^1_{6} (123)\},
\]
where
\[
\mathfrak{D}^1_{6} (123) = \{436215, 562143, 563421, 564213\}.
\]
\end{theorem}

Now, with respect to enumerations of Dumont-2 permutations avoiding patterns of length three, we can start with the makeup of a Dumont-2 permutation.  It is straightforward that for all $n \ge 3$, 
$$|\mathfrak{D}^2_{2n} (123)| = |\mathfrak{D}^2_{2n} (132)| = |\mathfrak{D}^2_{2n} (213)| = 0.$$  
This follows from the fact that every even position is a deficiency, every odd position is an excedance or a fixed point, and for any Dumont-2 permutation $\pi(2) = 1$ and $\pi(2n-1) = 2n$ or $2n-1$,  which eventually leads to the conclusion that it is impossible to avoid 123, 132, and 213.

\begin{theorem}[\cite{Burstein_Restricted}]
For all $n\ge 1$,
\[
|\mathfrak{D}^2_{2n} (231)| = 2^{n-1}.
\] 
\end{theorem}

\begin{theorem}[\cite{Burstein_Restricted}]
 For all $n\ge 0$,
\[
|\mathfrak{D}^2_{2n} (312)| = 1, \quad \text{namely} \quad \mathfrak{D}^2_{2n} (312) = \{2, 1, 4, 3, \dots, 2n, 2n-1\}.
\] 
\end{theorem}

\begin{theorem}[\cite{Mansour}]
For all $n\ge 0$,
\[
|\mathfrak{D}^2_{2n} (321)| = C_n = \frac{1}{n+1} \binom{2n}{n}.
\]
\end{theorem}

\subsubsection{Patterns of length four}
\label{length four_thms}

Now, for Dumont-1 and Dumont-2 permutations avoiding patterns of length four, there are several cases which are open, as well as several enumerations of permutations that avoid two patterns of length four simultaneously.  We begin with Dumont-2 permutations avoiding a single pattern of length four.  

\begin{theorem}[\cite{Burstein_Restricted}]
\label{3142_thm}
For all $n\ge 0$,
\[
|\mathfrak{D}^2_{2n} (3142)| = C_{n}.
\]
\end{theorem}

For pattern $4132$, note that $321$ is a subsequence of $4132$, therefore $\mathfrak{D}^2_{2n} (321) \subseteq \mathfrak{D}^2_{2n} (4132)$. The following theorem shows that, in fact, the two sets are equal.

\begin{theorem}[\cite{Burstein_Dyck}]
For all $n\ge 0$,
\[
|\mathfrak{D}^2_{2n} (4132)| = |\mathfrak{D}^2_{2n} (321)| = C_n.
\]
\end{theorem}

\begin{theorem}[\cite{Burstein_Dyck}]
For all $n\ge 0$,
\[
|\mathfrak{D}^2_{2n} (2143)| = a_n a_{n+1},
\]
where
\[
a_{2m} = \frac{1}{2m+1} \binom{3m}{m} \qquad \text{and} \qquad a_{2m+1} = \frac{1}{m+1} \binom{3m+1}{m}.
\]
\end{theorem}

\vspace{0.2in}

To date, there are no other enumerations of Dumont-1 and Dumont-2 permutations avoiding a single pattern of length four. However, we conjecture a Wilf-equivalence on Dumont-1 permutations of patterns $2143$ and $3421$ (see Section~\ref {B-J Conjecture}).

Now, we will consider simultaneous avoidance of a pair of patterns. We begin with Dumont-1 permutations.

\begin{theorem}[\cite{Burstein_Dyck}]
For all $n\ge 0$,
\[
|\mathfrak{D}^1_{2n} (1342, 1423)| = |\mathfrak{D}^1_{2n} (2341, 2413)| = |\mathfrak{D}^1_{2n} (1342, 2413)| =s_{n+1},
\]
where $s_n$ is the $n$-th little Schr\"oder number (\cite[A001003]{OEIS}).
\end{theorem}

Note that $(s_n)=(1, 1, 3, 11, 45, 197, 903, \dots)$ is recursively given by $s_1 = 1$ and
\[
s_{n+1} = -s_n + 2 \sum_{k=1}^n s_k s_{n-k}, \quad n\ge 2,
\]
with the generating function
\[
s(x) = \sum_{n=1}^\infty s_n x^n = \frac{1+x-\sqrt{1-6x+x^2}}{4}.
\]

\begin{theorem}[\cite{Burstein_Dyck}] 
For all $n\ge 1$,
\[
|\mathfrak{D}^1_{2n} (231, 4213)| = 1, \quad \text{namely,} \quad \mathfrak{D}^1_{2n} (231, 4213) = \{(2, 1, 4, 3, \dots 2n, 2n-1)\}.
\]
\end{theorem}

\begin{theorem}[\cite{Burstein_Dyck}] 
For all $n\ge 1$, $|\mathfrak{D}^1_{2n} (1342, 4213)| = 2^{n-1}$.
\end{theorem}

\begin{theorem}[\cite{Burstein_Dyck}] 
For all $n\ge 3$,
$|\mathfrak{D}^1_{2n} (2341, 1423)| = b_n$, where $b_n$ satisfies the recurrence relation 
\[
b_n = 3b_{n-1} + 2b_{n-2} \quad \text{for} \quad n \ge 3, \quad \text{with} \quad b_0 = 1, b_1 = 1, b_2 = 3.
\]
\end{theorem}

Note that the sequence $(b_n)=(1, 1, 3, 11, 39, 139, 495, \dots)$ is \cite[A007482]{OEIS} shifted one position to the right. In other words, $\mathfrak{D}^1_{2n} (2341, 1423)$ is equinumerous to the set of of subsets of $[2n-2]$ where each odd element $m$ has an even neighbor ($m-1$ or $m+1$).




\section{Avoidance on Dumont permutations of the fourth kind}

In this section, we will consider pattern avoidance on Dumont-4 permutations. In previous pattern-avoidance literature, the first nontrivial cases to be analyzed were patterns of length $3$. Since $1$ cannot be a deficiency in a Dumont-4 permutation, it must be a fixed point, so all Dumont-4 permutations start with $1$. Thus, we will also consider avoiding patterns $\pi= (1,\pi'+1)$, where $\pi'$ is a permutation in $S_3$, i.e. $\pi\in\{1234, 1243, 1324, 1342, 1423, 1432\}$. Note that in all but the first two cases, i.e. if $\pi'$ does not start with $1$, we have 
\[
\mathfrak{D}^4_{2n}(\pi)=\mathfrak{D}^4_{2n}(\pi').
\]

\subsection{Enumerating Dumont-4 permutations avoiding Dumont-4 permutations of length four}

We will begin by considering Dumont-4 permutations avoiding patterns of length $4$, which are themselves Dumont-4 permutations, that is $\pi\in\{1234, 1342, 1432\}$.

Then in the next section we will consider the remaining cases, where $\pi\in\{1243, 1324, 1423\}$.

\subsubsection{$\mathfrak{D}^4_{2n}(1234)$}

\begin{theorem}
$|\mathfrak{D}^4_{2n}(1234)| = 0$ \ for $n \ge 4$.
\end{theorem}

Note that $|\mathfrak{D}^4_{2n}(1234)| = 1, 1, 2, 4$, for $n=0, 1, 2, 3$, respectively. The Dumont-4 permutations of length at most 6 avoiding pattern 1234 are $\epsilon$, $12$, $1342$, $1432$, $132654$, $136254$, $143265$, $143652$, where $\epsilon$ denotes the empty permutation.

\begin{proof}
Recall that in Dumont-4 permutations, $\pi(1)=1$, and $\pi(2n-1)=2n-1$ or $2n$. In addition to that, for $n \ge 2$, we have $\pi(2)=3$ or $\pi(3)=3$. 

Since $\pi(1)=1$, the three conditions above mean that, of the $2n-5$ entries in $[4, 2n-2]$, at most one is to the left of $3$ and at most one is to the right of $\pi(2n-1)$. This leaves at least $2n-7\ge 1$ entries in $[4,2n-2]$ that are to the right of $3$ and to the left of $\pi(2n-1)$. Any such entry, together with $1$, $3$, and $\pi(2n-1)$ would form an occurrence of pattern $1234$.
\end{proof}

\subsubsection{$\mathfrak{D}^4_{2n}(1342)$}

As noted earlier, since the entry following $1$ is not $2$, it follows that $\mathfrak{D}^4_{2n}(1342)=\mathfrak{D}^4_{2n}(231)$.

\begin{theorem}
\label{1342_theorem}
$|\mathfrak{D}^4_{2n}(1342)| = 2^{n-1}$, \text{for} $n \ge 1$.
\end{theorem}

We will first prove that all odd entries are fixed points.

\begin{lemma}
\label{lm:odd_fixed}
Let  $\pi \in \mathfrak{D}^4_{2n}(1342)$.  Then $\pi(2k-1) = 2k-1$ for all $k$, $1 \le k \le n$.
\end{lemma}

\begin{proof}
By definition of a Dumont-4 permutation $\pi(1)=1$, in other words, $1$ is a fixed point. 


Now assume that all odd entries from $1$ through $2j-1$ are fixed points.  Consider the entry in the next odd position, $\pi(2j+1)$, and as before, suppose $\pi(2j+1) \ne 2j+1$.  It follows that 
\begin{itemize}
\item the entry $2j+1$ must be to the left of $\pi(2j+1)$, 
\item $\pi(2j+1)>2j+1$ as an odd position cannot be a deficiency, and 
\item at least one even entry $2l\le 2j$ must be to the right of $\pi(2j+1)$. This is because all odd entries at most $2j+1$ are to the left of $\pi(2j+1)$, so if all even entries less than $2j+1$ are also to the left of $2j+1$, then there will be $2j+1$ entries to the left of position $2j+1$, which is impossible.
\end{itemize}

This yields a $1342$-occurrence $(1, 2j+1, \pi(2j+1), 2l)$.  Therefore $\pi(2j+1)=2j+1$, that is $2j+1$ is fixed.

By induction, the lemma is proved.
\end{proof}

Next, we will prove that if we have a deficiency at position $2k$, then it must be the entry $2k-2$.

\begin{lemma}
\label{lm:def=2k-2}
Let $\pi \in \mathfrak{D}^4_{2n}(1342)$. For all k, $1 \le k \le n$, if $\pi(2k) < 2k$ then $\pi(2k) = 2k-2$.
\end{lemma}

\begin{proof}
We will prove by contradiction that no deficiency at position $2k$ can have an entry less than or equal to $2k-4$.  So, let $\pi(2k) \le 2k-4$.  Since all odd entries are fixed points by Lemma \ref{lm:odd_fixed}, the two odd entries $2k-3$ and $2k-1$ are above and to the left of $\pi(2k)$, the triple $(2k-3, 2k-1,\pi(2k))$ at positions $(2k-3,2k-1,2k)$ would yield a 231-occurrence. 

By contradiction, the lemma is proved.
\end{proof}

Now we will prove Theorem \ref {1342_theorem}.

\begin{proof}[of Theorem \ref{1342_theorem}]
Given Lemma \ref{lm:odd_fixed}, all odd entries are fixed. Thus the even values are the only entries of interest, and moreover, all even values occur in even positions.  Furthermore, due to Lemma \ref{lm:def=2k-2}, a deficiency at position $2k$ must be the entry $2k-2$.

For $\pi\in\mathfrak{D}^4_{2n}(1342)=\mathfrak{D}^4_{2n}(231)$ consider a permutation $\pi_e\in S_n(231)$ such that $\pi_e(i)=j$ if $\pi(2i)=2j$ for $i,j\in[n]$. Then, for each $i\in[n]$, we have $\pi_e(i)\ge i-1$. Suppose that $\pi_e(1)=k$, then 231 avoidance of $\pi_e$ implies that the entries in $[1,k-1]$ occupy positions $[2,k]$. But $\pi_e(i)\ge i-1$ for all $i\in[n]$, so this can only occur if the values $1,\dots,k-1$ are consecutive and in increasing order. In other words, $\pi_e=(k,1,2,\dots,k-1,\tau+k)$, where $\tau\in S_{n-k}(231)$ satisfies the same conditions as $\pi_e$. 

Note also that the only nondeficiency among the first $k$ letters is the value $k$. Thus, we see inductively that the entire structure of $\pi_e$ is determined by the values of its nondeficiencies, one of which must be the entry $n$. Combining $\pi_e$ as above with the odd entries of $\pi$ (all fixed points), we see that $\pi$ also avoids $231$. Thus, $\pi_e$ (and hence, $\pi$) is determined by the choice of a subset $S \subseteq [1,n-1]$ such that if $k\in S$, then $2k$ is a nondeficiency of $\pi$. All of these choices are unrestricted, therefore the number of such choices is $2^{n-1}$. Moreover, the semilengths of the resulting blocks of the same form as $(1,2k,\pi')$ yield a composition of $n$ into nonzero parts. For example, $\pi = 16325478 \in \mathfrak{D}^4_{8}(1342)$ corresponds to the composition $4=3+1$ (see Figure \ref{fig:many8x8}). \qedhere

\end{proof}



\begin{example} 
\label {ex: D^4_8(1342)}
Below are all eight Dumont-4 permutations of length $8$ avoiding $1342$, and the corresponding compositions of $4=8/2$.\\
\begin{center}

\begin{tikzpicture}[scale=0.4]

\begin{scope}
\draw[scale = 0.25, step = 2, very thin,color=gray] (0,0) grid (16,16);
 \foreach \position in {(1,1), (5,5), (9,9), (13,13)} 
    \filldraw [scale = 0.25, color=black] \position circle (8pt);
 \foreach \position in {(3,3), (7,7), (11,11), (15,15)} 
    \filldraw [scale = 0.25, color=red] \position circle (8pt);
\draw [line width = 1pt, color=blue] (0,0) rectangle (1,1);
\draw [line width = 1pt, color=blue] (1,1) rectangle (2,2);
\draw [line width = 1pt, color=blue] (2,2) rectangle (3,3);
\draw [line width = 1pt, color=blue] (3,3) rectangle (4,4);
\node [scale=0.6, below right] (0,0) {$4=1+1+1+1$};
\end{scope}

\begin{scope}[xshift=6cm]
\draw[scale = 0.25, step = 2, very thin,color=gray] (0,0) grid (16,16);
 \foreach \position in {(1,1), (5,5), (9,9), (13,13)} 
    \filldraw [scale = 0.25, color=black] \position circle (8pt);
 \foreach \position in {(3,3), (7,7), (11,15), (15,11)} 
    \filldraw [scale = 0.25, color=red] \position circle (8pt);
\draw [line width = 1pt, color=blue] (0,0) rectangle (1,1);
\draw [line width = 1pt, color=blue] (1,1) rectangle (2,2);
\draw [line width = 1pt, color=blue] (2,2) rectangle (4,4);
\node [scale=0.6, below right] (0,0) {$4=1+1+2$};
\end{scope}

\begin{scope}[xshift=12cm]
\draw[scale = 0.25, step = 2, very thin,color=gray] (0,0) grid (16,16);
 \foreach \position in {(1,1), (5,5), (9,9), (13,13)} 
    \filldraw [scale = 0.25, color=black] \position circle (8pt);
 \foreach \position in {(3,3), (7,11), (11,7), (15,15)} 
    \filldraw [scale = 0.25, color=red] \position circle (8pt);
\draw [line width = 1pt, color=blue] (0,0) rectangle (1,1);
\draw [line width = 1pt, color=blue] (1,1) rectangle (3,3);
\draw [line width = 1pt, color=blue] (3,3) rectangle (4,4);
\node [scale=0.6, below right] (0,0) {$4=1+2+1$};
\end{scope}

\begin{scope}[xshift=18cm]
\draw[scale = 0.25, step = 2, very thin,color=gray] (0,0) grid (16,16);
 \foreach \position in {(1,1), (5,5), (9,9), (13,13)} 
    \filldraw [scale = 0.25, color=black] \position circle (8pt);
 \foreach \position in {(3,3), (7,15), (11,7), (15,11)} 
    \filldraw [scale = 0.25, color=red] \position circle (8pt);
\draw [line width = 1pt, color=blue] (0,0) rectangle (1,1);
\draw [line width = 1pt, color=blue] (1,1) rectangle (4,4);
\node [scale=0.6, below right] (0,0) {$4=1+3$};
\end{scope}

\begin{scope}[xshift=0cm, yshift=-6cm]
\draw[scale = 0.25, step = 2, very thin,color=gray] (0,0) grid (16,16);
 \foreach \position in {(1,1), (5,5), (9,9), (13,13)} 
    \filldraw [scale = 0.25, color=black] \position circle (8pt);
 \foreach \position in {(3,7), (7,3), (11,11), (15,15)} 
    \filldraw [scale = 0.25, color=red] \position circle (8pt);
\draw [line width = 1pt, color=blue] (0,0) rectangle (2,2);
\draw [line width = 1pt, color=blue] (2,2) rectangle (3,3);
\draw [line width = 1pt, color=blue] (3,3) rectangle (4,4);
\node [scale=0.6, below right] (0,0) {$4=2+1+1$};
\end{scope}

\begin{scope}[xshift=6cm, yshift=-6cm]
\draw[scale = 0.25, step = 2, very thin,color=gray] (0,0) grid (16,16);
 \foreach \position in {(1,1), (5,5), (9,9), (13,13)} 
    \filldraw [scale = 0.25, color=black] \position circle (8pt);
 \foreach \position in {(3,7), (7,3), (11,15), (15,11)} 
    \filldraw [scale = 0.25, color=red] \position circle (8pt);
\draw [line width = 1pt, color=blue] (0,0) rectangle (2,2);
\draw [line width = 1pt, color=blue] (2,2) rectangle (4,4);
\node [scale=0.6, below right] (0,0) {$4=2+2$};
\end{scope}

\begin{scope}[xshift=12cm, yshift=-6cm]
\draw[scale = 0.25, step = 2, very thin,color=gray] (0,0) grid (16,16);
 \foreach \position in {(1,1), (5,5), (9,9), (13,13)} 
    \filldraw [scale = 0.25, color=black] \position circle (8pt);
 \foreach \position in {(3,11), (7,3), (11,7), (15,15)} 
    \filldraw [scale = 0.25, color=red] \position circle (8pt);
\draw [line width = 1pt, color=blue] (0,0) rectangle (3,3);
\draw [line width = 1pt, color=blue] (3,3) rectangle (4,4);
\node [scale=0.6, below right] (0,0) {$4=3+1$};
\end{scope}

\begin{scope}[xshift=18cm, yshift=-6cm]
\draw[scale = 0.25, step = 2, very thin,color=gray] (0,0) grid (16,16);
 \foreach \position in {(1,1), (5,5), (9,9), (13,13)} 
    \filldraw [scale = 0.25, color=black] \position circle (8pt);
 \foreach \position in {(3,15), (7,3), (11,7), (15,11)} 
    \filldraw [scale = 0.25, color=red] \position circle (8pt);
\draw [line width = 1pt, color=blue] (0,0) rectangle (4,4);
\node [scale=0.6, below right] (0,0) {$4=4$};
\end{scope}

\end{tikzpicture}


%
%


%
%


%
%


\begin{figure}[ht]
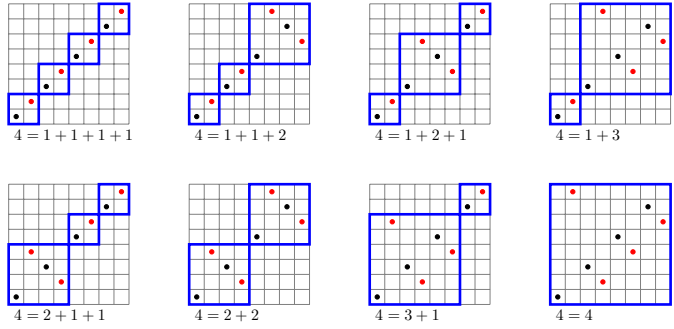

 \caption{Dumont-4 permutations of length $8$ avoiding $1342$.}
\label{fig:many8x8}
\end{figure}
\end{center}

\end{example}



\subsubsection{$\mathfrak{D}^4_{2n}(1432)$ \label{321Dyck}}

As noted earlier, since the entry following $1$ is not $2$, it follows that $\mathfrak{D}^4_{2n}(1432)=\mathfrak{D}^4_{2n}(321)$.

\begin{theorem}
\label{thm:1432}
$|\mathfrak{D}^4_{2n}(1432)|=|\mathfrak{D}^4_{2n}(321)|=C_n$ for $n \ge 0$, where $C_n$ is the $n$-th Catalan number.
\label{321thm}
\end{theorem}

Since we are considering $321$-avoiding permutations, we will adapt a bijection of Krattenthaler \cite{Krattenthaler} from $S_n(321)$ to Dyck paths of semilength $n$.

\begin{proof}
Let $\pi \in \mathfrak{D}^4_{2n}(1432) = \mathfrak{D}^4_{2n}(321)$, and consider right-to-left minima (an entry $\pi(i)$ for which $\pi(i) < \pi(j)$ whenever $i < j$).  If an odd fixed point is a right-to-left minimum, then every entry to the left must have a smaller value, and every entry to the right must have a larger value, otherwise we encounter a 321-occurrence with the entry $2k+1$ serving as the ``2''.  In other words, if $\pi(2k+1) = 2k+1$, for some $0 \le k \le n-1$, then $\pi = (\pi', 2k+1, \pi''+2k+1)$ where $\pi' \in \mathfrak{D}^4_{2k}(321)$, and $(1,\pi''+1) \in \mathfrak{D}^4_{2n-2k}(321)$.  And, since the odd entry $\pi(2k+1) = 2k+1$ is a right-to-left minimum, the entry $2k+2$, which must be to the right of $2k+1$, is also a right-to-left minimum.

Consider a $2n \times 2n$ board with the dots in the $i$-th column from the left being in $\pi(i)$-th row from the bottom, for $1\le i\le 2n$. Now consider the dots which represent even right-to-left minima (solid red dots in Figure \ref{fig:big12x12_1432}), and lower them one cell down (hollow red dots in Figure \ref{fig:big12x12_1432}). Notice that, in particular, this will place a hollow red dot in each row with an odd right-to-left minimum.

Now travel along the cell boundaries from $(0, 0)$ to $(2n, 2n)$ in an \emph{East-North} fashion, using steps $(1,0)$ (\emph{east}) and $(0,1)$ (\emph{north}) and keeping all dots (both filled and hollow) to the left of the path, staying as close to the diagonal as possible. Equivalently, this is the path $P$ where the peaks (instances where an east step is followed directly by a north step) are exactly the bottom and right boundaries of the cells with hollow red dots.

For example, consider Figure \ref{fig:big12x12_1432} corresponding to $\pi = 1\ 3\ 5\ 2\ 6\ 4\ 7\ 8\ 9\ 11\ 12\ 10 \in \mathfrak{D}^4_{12}(321)$. Note that red dots are even right-to-left minima, the blue dots are odd right-to-left minima (which are fixed points), and the black dots are excedances.  

Suppose that $\pi$ has $k$ even non-excedances with (even) values $2=b_1<b_2<\dots<b_k$ at (even) positions $a_1<a_2<\dots<a_k=2n$, respectively. Also, let $b_{k+1}=2n+2$ and $a_0=0$. Then the \emph{runs} (maximal blocks) of east steps have lengths $a_i-a_{i-1}$ for $1\le i\le k$, and the runs of north steps have lengths $b_{i+1}-b_i$ for $1\le i\le k$.

Therefore, all runs of east and north steps in path $P$ are of even length (see Figure \ref{fig:big12x12_1432}, with $n=6$ and even non-excedance values of $2,4,8,10$ at positions $4,6,8,12$). Dividing the lengths of these runs in half, we obtain a Dyck path of semilength $n$ from $(0,0)$ to $(n,n)$.


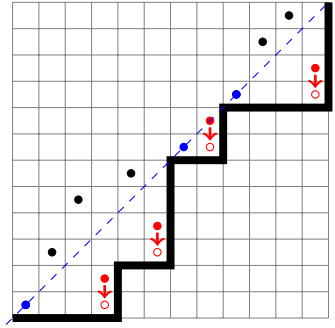
\begin{figure}[!h]
\begin{center}
\begin{tikzpicture}[scale=0.175]
\draw[step = 2, very thin,color=gray] (0,0) grid (24,24);
\foreach \position in {(1,1), (3,5), (5,9), (9,11), (17,17), (19,21), (21, 23)} 
    \filldraw [color = black] \position circle (8pt);
\foreach \position in {(1,1), (13,13), (17,17)} 
    \filldraw [color = blue] \position circle (8pt);
\foreach \position in {(7,3), (11,7), (15,15), (23, 19)} 
    \filldraw [color = red] \position circle (8pt);
\foreach \position in {(7,1), (11,5), (15,13), (23, 17)} 
    \draw [color = red] \position circle (8pt);
\draw [->, line width = 1pt, color=red] (7,2.5) -- (7,1.5);
\draw [->, line width = 1pt, color=red] (11,6.5) -- (11,5.5);
\draw [->, line width = 1pt, color=red] (15,14.5) -- (15,13.5);
\draw [->, line width = 1pt, color=red] (23,18.5) -- (23,17.5);
\draw [line width = 3pt, color=black] (0,0) -- (8,0) -- (8,4) -- (12, 4) -- (12,12) -- (16,12) -- (16,16) -- (24,16) -- (24,24);
\draw [dashed, color = blue] (-0.5,-0.5) -- (24.5,24.5);
\end{tikzpicture} 
\caption{Permutation diagram of $\pi = 1\ 3\ 5\ 2\ 6\ 4\ 7\ 8\ 9\ 11\ 12\ 10 \in \mathfrak{D}^4_{12}(1432)$} 
\label{fig:big12x12_1432}
\end{center}
\end{figure}

Therefore, the number of Dumont-4 permutations of length $2n$ avoiding $321$ is the same as the number of Dyck paths of semilength $n$, i.e. the $n$th Catalan number, $C_n$.  This ends the proof.
\end{proof}




\subsection {Enumerating Dumont-4 permutations avoiding certain permutations of length four}

Now that all three Dumont-4 permutations of length four have been avoided by Dumont-4 permutations of length $2n$, we will look at three other permutations of length four starting with $1$, namely $1324$, $1243$, and $1423$.  Also note, excluding the entry ``$1$", the three permutations in the previous section coupled with the three permutations in this section constitute all of $S_3$.

\subsubsection{$\mathfrak{D}^4_{2n}(1324)$}

As noted earlier, since the entry following $1$ is not $2$, it follows that $\mathfrak{D}^4_{2n}(1324)=\mathfrak{D}^4_{2n}(213)$.

\begin{theorem}
$|\mathfrak{D}^4_{2n}(1324)| = 2 \binom{n}{2} + 1 = n^2 - n + 1$, \text{for} $n \ge 0$.
\end{theorem}

\begin{proof}
Consider the rightmost value of $\pi\in\mathfrak{D}^4_{2n}(1324)$. If the entry $2n$ occurs in the last position, then $\pi = 1,2,3,4,\dots,2n-1,2n$.  This permutation accounts for the ``$1$" in $2\binom{n}{2} + 1$.

If the last entry is not $2n$, then it must be a deficiency, so by definition of a Dumont-4 permutation, it must be even, say $2k$, where $1 \le k \le n-1$. 

Note that $\pi(2n-1)=2n-1$ or $\pi(2n-1)=2n$, and therefore all entries below and to the left of $\pi(2n-1)$ must be in increasing order to avoid pattern $213$. Therefore, all entries in $[1,2n-1]\setminus\{2k\}$, i.e. all entries of $\pi$ except $2n$ and $2k$ are in increasing order.

Moreover, the value $2k-1$ can only occur in position $2k-1$ or smaller, so for the values $1,2,\dots,2k-1$ to form an increasing subsequence of $\pi$, all of them must be fixed points. Thus, the value $2n$ must occur in some position $l\in[2k,2n-1]$. In fact, each choice of $\pi(2n)=2k$ and $\pi^{-1}(2n)=l\in[2k,2n-1]$ yields a unique permutation $\pi\in\mathfrak{D}^4_{2n}(1324)$.

Thus, 
we can write $\pi=(1,\pi',\rho,2n,\sigma,2k)$, where $\pi'= (1, 2, \dots, 2k-1)$, $\rho=(2k+1, 2k+2, \dots, l)$, and $\sigma=(l+1,l+2, \dots, 2n-1)$.

When $\pi(2n)=2k<2n$, the number of possible choices for the position $l$ of $2n$ is the number of elements in $[2k,2n-1]$, i.e. $2n-2k$. Since $1\le k\le n-1$, the total number of Dumont-4 permutations avoiding 1324 is
\[
|\mathfrak{D}^4_{2n}(1324)| = 1 + \sum_{k=1}^{n-1}(2n-2k)
= 1 + 2\sum_{k=1}^{n-1}(n-k)
= 1 + 2\sum_{k=1}^{n-1}k
= 1+ 2 \binom {n}{2}
= n^2-n+1. \qedhere
\]
\end{proof}

\begin{example}
\label{example:D^4_16(1324)}
A Dumont-4 permutation of length $16$ avoiding $1324$ can be obtained by placing the entry $6$ in position $16$, and the entry $16$ in position $10$. All other entries must be in increasing order.  
See Figure \ref{fig:big16x16} for an example.

\begin{figure}[!ht]
\begin{center}
\begin{tikzpicture}[scale = 0.4]
\draw[very thin, color = gray, scale = 0.7] (0,0) grid (16,16);
 \foreach \position in {(1,1), (3,3), (5,5), (7,7), (9,9), (11,13), (13,15), (15,17), (17,19), (21,21), (23,23), (25, 25), (27, 27), (29, 29)}
    \filldraw [scale = 0.35, color = black] \position circle (8pt);
 \foreach \position in {(31,11)} 
    \filldraw [scale = 0.35, color = red] \position circle (8pt);
 \foreach \position in {(19,31)} 
    \filldraw [scale = 0.35, color = blue] \position circle (8pt);
\draw [line width = 2pt, color = red, scale = 0.7] (1,1) rectangle (5,5);
\draw [line width = 2pt, color = red, scale = 0.7] (5,6) rectangle (15,16);
\draw [line width = 2pt, color = blue, dashed, scale = 0.7] (5.1,6.1) rectangle (9,10);
\draw [line width = 2pt, color = blue, dashed, scale = 0.7] (10, 10) rectangle (14.9,15);
\end{tikzpicture} 
	\caption {Permutation diagram of
	$\pi = 	1\ 2\ 3\ 4\ 5\ 7\ 8\ 9\ 10\ 16\ 11\ 12\ 13\ 14\ 15\ 6 
			\in \mathfrak{D}^4_{16}(1324)$}
	\label{fig:big16x16}
\end{center}
\end{figure}
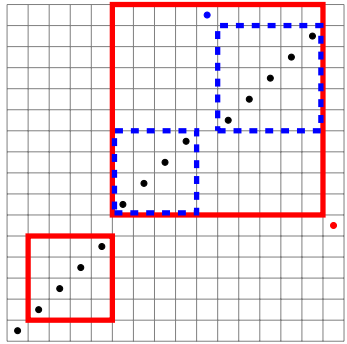

\end{example}

\subsubsection{$\mathfrak{D}^4_{2n}(1243)$}

\begin{theorem} \label{thm:1243}
$|\mathfrak{D}^4_{2n}(1243)| = |\mathfrak{D}^4_{2n}(1324)| = 2 \binom{n}{2} + 1 = n^2 - n + 1$, \text{for} $n \ge 0$.
\end{theorem}

In other words, patterns $1324$ and $1243$ are Wilf-equivalent on Dumont permutations of the fourth kind.

\begin{proof}
This enumeration is the same as that of $\mathfrak{D}^4_{2n}(1324)$ since the removal of the entry $1$ from patterns $1324$ and $1243$ yields the patterns $213$ and $132$ that are each other's reflections with respect to the antidiagonal.  That is, $213$ and $132$ are inverses of reversals of complements of each other, and thus their respective avoidance classes are enumerated by the same sequence.

To obtain a permutation in $\mathfrak{D}^4_{2n}(1243)$ from a permutation in $\mathfrak{D}^4_{2n}(1324)$, remove the entry ``$1$" and reflect the remaining entries from  about the antidiagonal. This maps blocks onto blocks, with their diagonals mapping onto the diagonals of the images of those blocks.  Lastly, add $1$ to every entry and prepend the value $1$ that was removed. See Figure \ref{fig:big16x16_irc} for an example. Note that reflection about the antidiagonal applied above implies that any permutation $\pi\in\mathfrak{D}^4_{2n}(1243)$ is uniquely given by $\pi^{-1}(2)$ and $\pi(2)\in[2,\pi^{-1}(2)]$.
\end{proof}

\begin{example}
\label{example:D^4_16(1243)}
Reflecting the 1324-avoiding permutation in Figure \ref{fig:big16x16} as in Theorem \ref{thm:1243}, we obtain a 1243-avoiding Dumont-4 permutation.
%
%
\begin{figure}[!h]
\begin{center}
\begin{tikzpicture}[scale = 0.4]
\draw[very thin, color = gray, scale = 0.7] (0,0) grid (16,16);
 \foreach \position in {(1,1), (5,5), (7,7), (9,9), (11,11), (13,13), (15,17), (17,19), (19,21), (21,23), (25, 25), (27, 27), (29, 29), (31,31)}
    \filldraw [scale = 0.35, color = black] \position circle (8pt);
 \foreach \position in {(23,3)} 
    \filldraw [scale = 0.35, color = red] \position circle (8pt);
 \foreach \position in {(3,15)} 
    \filldraw [scale = 0.35, color = blue] \position circle (8pt);
\draw [line width = 2pt, color = red, scale = 0.7] (12,12) rectangle (16,16);
\draw [line width = 2pt, color = red, scale = 0.7] (1,2) rectangle (11,12);
\draw [line width = 2pt, color = blue, dashed, scale = 0.7] (2.1,2.1) rectangle (7,7);
\draw [line width = 2pt, color = blue, dashed, scale = 0.7] (7, 8) rectangle (10.9,11.9);
\end{tikzpicture} 
	\caption {Permutation diagram of
	$\rho = 	1\ 8\ 3\ 4\ 5\ 6\ 7\ 9\ 10\ 11\ 12\ 2\ 13\ 14\ 15\ 16 		
					\in \mathfrak{D}^4_{16}(1243)$.}
	\label{fig:big16x16_irc}
\end{center}
\end{figure}
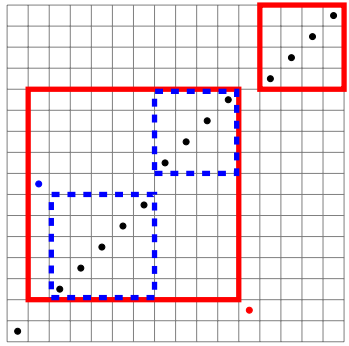
\end{example}

\subsubsection{$\mathfrak{D}^4_{2n}(1423)$}

As noted earlier, since the entry following $1$ is not $2$, it follows that $\mathfrak{D}^4_{2n}(1423)=\mathfrak{D}^4_{2n}(312)$.

Since our result for this pattern involves a continued fraction, we need a remark on notation.

\begin{notation}
For convenience, the following notation will be used for continued fractions:
\[
\cfrac{a}{\alpha \pm \cfrac{b}{\beta \pm \cfrac{c}{\gamma \pm \ddots}}} \quad = \quad 
\frac{a}{\alpha} \,\, \genfrac{}{}{0pt}{}{}{\pm} \,\, \frac{b}{\beta} \,\, \genfrac{}{}{0pt}{}{}{\pm} \,\, \frac{c}{\gamma} \,\, \genfrac{}{}{0pt}{}{}{\pm} \,\, \cdots
\]
\end{notation}


Additionally, we will need to use the truncations of the even and odd parts of the Catalan generating function $C(z)=\sum_{n=0}^{\infty}{C_n z^n}$. Define 
\[
C_{e,2m}=C_{e,2m}(z) =  \sum_{i=0}^{m} C_{2i}z^{2i} 
\quad \text{and} \quad 
C_{o,2m+1}=C_{o,2m+1}(z) =  \sum_{i=0}^{m} C_{2i+1}z^{2i+1},
\]
where we set these functions equal to $0$ when $m<0$.

\begin{theorem}
For all $k \ge 1$, the ordinary generating function for $|\mathfrak{D}^4_{2n}(1423)|$ is $(R_1/z)\circ\sqrt{z}=(zR_1/z^2)\circ\sqrt{z}$, where the sequence of functions $zR_{2k+1}$, $k\ge 0$, satisfies the recurrence relation
\begin{equation} \label{eq:d4-312-ogf}
zR_{2k+1} = \frac{z^2C_{e,2k}}{(1-zC_{o,2k-1})^2} \,\,
 \genfrac{}{}{0pt}{}{}{-} \,\,
\frac{z^2C_{e,2k}}{1-zC_{o,2k+1}} \,\,
 \genfrac{}{}{0pt}{}{}{-} \,\,
\frac{z^2C_{e,2k}^2}{1} \,\,
 \genfrac{}{}{0pt}{}{}{-} \, \,
zR_{2k+3}.
\end{equation} 
\end{theorem}

This yields the generating function for $|\mathfrak{D}^4_{2n}(1423)|$ in the form of a continued fraction.


We consider a Dumont-4 permutation avoiding $1423$, and we are interested in the parity of the smallest entry in each block using block decomposition after the ``$1$''.  Note, since every Dumont-4 permutation begins with the entry ``$1$'', avoiding $1423$ is the same as avoiding $312$ by the Dumont-4 permutation with the ``$1$'' removed from the beginning of the permutation.  Therefore we analyze the block decomposition of the resulting permutation diagram, ignoring the initial block of ``$1$''.

To analyze the blocks resulting from the iterations of the block decompositon, we will need an auxiliary parameter, namely, the length $m$ of the maximal contiguous segment of allowed cells in the bottom row. Our cases are further subdivided according to the parity of $m$ and $n-m$, where $n$ is the dimension of the (square) block. We will refer to the cells on the line $y=x-m$ as the \emph{$m$-th subdiagonal} of a board and call the cells below the $m$-th subdiagonal \emph{$m$-deficiencies} (so, deficiencies as defined earlier are 0-deficiencies in this terminology).

Now, if the bottom row of a block is an odd row in the starting diagram of a Dumont-4 permutation then there is no $m$-deficiency in that bottom row.  If the bottom row of a block is an even row in the starting diagram, then there may be an $m$-deficiency in that bottom row.  Given the definition of the Dumont-4 permutations, in the boards resulting from the repeated block decomposition, all cells  above or to the left of the $m$-th subdiagonal are allowed, whereas the positions and values of the possible $m$-deficiencies are parity-restricted as described below:
\begin{description}
\item[EE blocks:] $m=2k$, $n-m$ is even (so $n$ is even), and all $m$-deficiencies are even values in odd positions;
\item[NE blocks:] $m=2k+1$, $n-m$ is even (so $n$ is odd), and all $m$-deficiencies are even values in even positions;
\item[EN blocks:] $m=2k$, $n-m$ is odd (so $n$ is odd), and all $m$-deficiencies are odd values in even positions;
\item[NN blocks:] $m=2k+1$, $n-m$ is odd (so $n$ is even), and all $m$-deficiencies are odd values in odd positions.
\end{description}

We denote the blocks mnemonically using ``E'' for even and ``N'' for ``not even,'' i.e. odd. However, for convenience in working with generating functions, we also let the $EE$ (resp. $NE, EN$, and $NN$) block with parity-restricted positions and values of $m$-deficiencies be represented by the generating function $P_m=P_m(z)$ (resp. $R_m=R_m(z),S_m=S_m(z)$, and $T_m=T_m(z)$), and refer to that block as a $P$-board (resp. $R$-board, $S$-board, and $T$-board).  See Figure \ref{312_odd_even row}, where blue dots are odd positions and red dots are even positions in the starting $\mathfrak{D}^4$-permutation diagram.



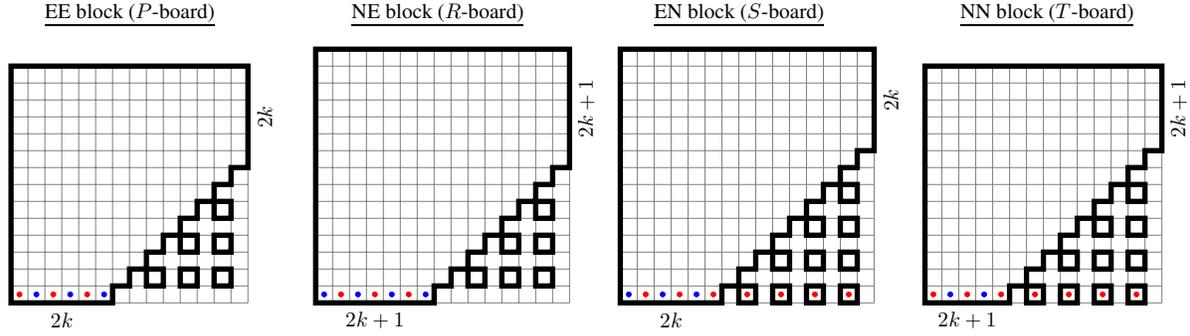
\begin{figure}[!ht]
\begin{center}
\begin{tikzpicture}[scale = 0.45]
\newcommand{\mysize}{0.5}

\begin{scope}
\draw[very thin, scale = \mysize, color = gray] (0,0) grid (14,14);
\foreach \x in {6} \foreach \y in {0}
\draw [line width = 2pt, scale = \mysize, color = black] (0,0) -- (\x,\y) -- (\x,\y+1) -- (\x+1,\y+1) -- (\x+1,\y+2) -- (\x+2,\y+2) -- (\x+2,\y+3) -- (\x+3,\y+3) -- (\x+3,\y+4) -- (\x+4,\y+4) -- (\x+4,\y+5) -- (\x+5,\y+5) -- (\x+5,\y+6) -- (\x+6,\y+6) -- (\x+6,\y+7) -- (\x+7,\y+7) -- (\x+7,\y+8) -- (\x+8,\y+8) -- (\x+8,\y+14) -- (0,\y+14) -- cycle;
\foreach \x in {8,10,12} \foreach \y in {1}
{\draw [line width = 2pt, scale = 0.5, color = black] (\x,\y) rectangle (\x+1, \y+1);}
\foreach \x in {10,12} \foreach \y in {3}
{\draw [line width = 2pt, scale = 0.5, color = black] (\x,\y) rectangle (\x+1, \y+1);}
\foreach \x in {12} \foreach \y in {5}
{\draw [line width = 2pt, scale = 0.5, color = black] (\x,\y) rectangle (\x+1, \y+1);}
\foreach \position in {(0.5,0.5), (2.5, 0.5), (4.5,0.5)} 
    \filldraw [scale = \mysize, color = red] \position circle (4pt);
 \foreach \position in {(1.5,0.5), (3.5, 0.5), (5.5,0.5)} 
    \filldraw [scale = \mysize, color = blue] \position circle (4pt);
\node[scale=0.8] at (1.5,-0.5) {$2k$};
\node[scale=0.8, rotate=90] at (7.5, 5.5) {$2k$};
\node[scale=0.8] at (3.5, 8.5) {\underline{EE block ($P$-board)}};
\end{scope}

\begin{scope}[xshift=9cm]
\draw[very thin, scale = \mysize, color = gray] (0,0) grid (15,15);
\foreach \x in {7} \foreach \y in {0}
\draw [line width = 2pt, scale = \mysize, color = black] (0,0) -- (\x,\y) -- (\x,\y+1) -- (\x+1,\y+1) -- (\x+1,\y+2) -- (\x+2,\y+2) -- (\x+2,\y+3) -- (\x+3,\y+3) -- (\x+3,\y+4) -- (\x+4,\y+4) -- (\x+4,\y+5) -- (\x+5,\y+5) -- (\x+5,\y+6) -- (\x+6,\y+6) -- (\x+6,\y+7) -- (\x+7,\y+7) -- (\x+7,\y+8) -- (\x+8,\y+8) -- (\x+8,\y+15) -- (0,\y+15) -- cycle;
\foreach \x in {9,11,13} \foreach \y in {1}
{\draw [line width = 2pt, scale = \mysize, color = black] (\x,\y) rectangle (\x+1, \y+1);}
\foreach \x in {11,13} \foreach \y in {3}
{\draw [line width = 2pt, scale = \mysize, color = black] (\x,\y) rectangle (\x+1, \y+1);}
\foreach \x in {13} \foreach \y in {5}
{\draw [line width = 2pt, scale = \mysize, color = black] (\x,\y) rectangle (\x+1, \y+1);}
\foreach \position in {(0.5,0.5), (2.5, 0.5), (4.5,0.5), (6.5,0.5)} 
    \filldraw [scale = \mysize, color = blue] \position circle (4pt);
 \foreach \position in {(1.5,0.5), (3.5, 0.5), (5.5,0.5)} 
    \filldraw [scale = \mysize, color = red] \position circle (4pt);
\node[scale=0.8] at (1.75,-0.5) {$2k+1$};
\node[scale=0.8, rotate=90] at (8, 5.75) {$2k+1$};
\node[scale=0.8] at (3.6, 8.5) {\underline{NE block ($R$-board)}};
\end{scope}

%
%

\begin{scope}[xshift=18cm]
\draw[very thin, scale = \mysize, color = gray] (0,0) grid (15,15);
\foreach \x in {6} \foreach \y in {0}
\draw [line width = 2pt, scale = \mysize, color = black] (0,0) -- (\x,\y) -- (\x,\y+1) -- (\x+1,\y+1) -- (\x+1,\y+2) -- (\x+2,\y+2) -- (\x+2,\y+3) -- (\x+3,\y+3) -- (\x+3,\y+4) -- (\x+4,\y+4) -- (\x+4,\y+5) -- (\x+5,\y+5) -- (\x+5,\y+6) -- (\x+6,\y+6) -- (\x+6,\y+7) -- (\x+7,\y+7) -- (\x+7,\y+8) -- (\x+8,\y+8) -- (\x+8,\y+9) -- (\x+9,\y+9) -- (\x+9,\y+15) -- (0,\y+15) -- cycle;
\foreach \x in {7,9,11,13} \foreach \y in {0}
{\draw [line width = 2pt, scale = \mysize, color = black] (\x,\y) rectangle (\x+1, \y+1);}
\foreach \x in {9,11,13} \foreach \y in {2}
{\draw [line width = 2pt, scale = \mysize, color = black] (\x,\y) rectangle (\x+1, \y+1);}
\foreach \x in {11, 13} \foreach \y in {4}
{\draw [line width = 2pt, scale = \mysize, color = black] (\x,\y) rectangle (\x+1, \y+1);}
\foreach \x in {13} \foreach \y in {6}
{\draw [line width = 2pt, scale = \mysize, color = black] (\x,\y) rectangle (\x+1, \y+1);}
\foreach \position in {(0.5,0.5), (2.5, 0.5), (4.5,0.5)} 
    \filldraw [scale = \mysize, color = blue] \position circle (4pt);
 \foreach \position in {(1.5,0.5), (3.5, 0.5), (5.5,0.5), (7.5, 0.5), (9.5, 0.5), (11.5, 0.5), (13.5,0.5)} 
    \filldraw [scale = \mysize, color = red] \position circle (4pt);
\node[scale=0.8] at (1.5,-0.5) {$2k$};
\node[scale=0.8, rotate = 90] at (8, 6) {$2k$};
\node[scale=0.8] at (3.5, 8.5) {\underline{EN block ($S$-board)}};
\end{scope}

\begin{scope}[xshift=27cm]
\draw[very thin, scale = \mysize, color = gray] (0,0) grid (14,14);
\foreach \x in {5} \foreach \y in {0}
\draw [line width = 2pt, scale = \mysize, color = black] (0,0) -- (\x,\y) -- (\x,\y+1) -- (\x+1,\y+1) -- (\x+1,\y+2) -- (\x+2,\y+2) -- (\x+2,\y+3) -- (\x+3,\y+3) -- (\x+3,\y+4) -- (\x+4,\y+4) -- (\x+4,\y+5) -- (\x+5,\y+5) -- (\x+5,\y+6) -- (\x+6,\y+6) -- (\x+6,\y+7) -- (\x+7,\y+7) -- (\x+7,\y+8) -- (\x+8,\y+8) -- (\x+8,\y+9) -- (\x+9,\y+9) -- (\x+9,\y+14) -- (0,\y+14) -- cycle;
\foreach \x in {6,8,10,12} \foreach \y in {0}
{\draw [line width = 2pt, scale = \mysize, color = black] (\x,\y) rectangle (\x+1, \y+1);}
\foreach \x in {8,10,12} \foreach \y in {2}
{\draw [line width = 2pt, scale = \mysize, color = black] (\x,\y) rectangle (\x+1, \y+1);}
\foreach \x in {10,12} \foreach \y in {4}
{\draw [line width = 2pt, scale = \mysize, color = black] (\x,\y) rectangle (\x+1, \y+1);}
\foreach \x in {12} \foreach \y in {6}
{\draw [line width = 2pt, scale = \mysize, color = black] (\x,\y) rectangle (\x+1, \y+1);}
\foreach \position in {(0.5,0.5), (2.5, 0.5), (4.5,0.5), (6.5,0.5), (8.5, 0.5), (10.5,0.5), (12.5, 0.5)} 
    \filldraw [scale = \mysize, color = red] \position circle (4pt);
 \foreach \position in {(1.5,0.5), (3.5,0.5)} 
    \filldraw [scale = \mysize, color = blue] \position circle (4pt);
\node[scale=0.8] at (1.25,-0.5) {$2k+1$};
\node[scale=0.8, rotate=90] at (7.5, 5.75) {$2k+1$};
\node[scale=0.8] at (3.6, 8.5) {\underline{NN block ($T$-board)}};
\end{scope}

\end{tikzpicture} 
\end{center}

\caption{Blocks whose bottom row was odd (even) in the starting $\mathfrak{D}^4$-permutation diagram}
\label{312_odd_even row}
\end{figure}

\begin{proof}
We now produce the generating function for the enumeration sequence of $\mathfrak{D}^4_{2n} (1423) = \mathfrak{D}^4_{2n} (312)$.  To begin, note that a $P$-board and $T$-board may be empty, while an $R$-board and $S$-board cannot be empty.  This is due to the fact that $P$-boards and $T$-boards have even dimensions whereas  $R$-boards and $S$-boards have odd dimensions. 

Consider the recurrence relations for the generating functions corresponding to each of the four blocks. We use the position of the entry $1$ in the block to produce those. Since a permutation in each block avoids $312$, any value to the left of $1$ must be less that any value to the right of $1$. For the $P$-board with $m=2k$, $k\ge 1$, we have the generating function $P_{2k}(z)$, the recurrence relation for which has three summands:
\begin{itemize}
\item The $P$-board may be empty, which corresponds to the summand ``$1$".  
\item If the $P$-board is nonempty, consider the entry in the bottom row. If it is in an odd position $\le 2k-1$, the factor of $z$ corresponds to the bottom row entry, in this case a blue dot on the $P$-board in Figure~\ref{312_odd_even row}. The block to the left of the blue dot must be a square of even dimension $\le 2k-2$ with all cells allowed, and the block to the right of the blue dot is an $S$-board with $m=2k$, correponding to the generating function $S_{2k}$.  This yields the summand ${\color{red}z C_{e,2k-2} S_{2k}}$.
\item Lastly, consider the entry in the bottom row that is in an even position $\le 2k$.  Again, the factor $z$ corresponds to the bottom row entry, in this case is a red dot on the $P$-board in Figure~\ref{312_odd_even row}. The block to the left of the red dot must be a square of odd dimension $\le 2k-1$ with all cells allowed, and the block to the right of the red dot is a $P$-board with $m=2k$, corresponding to the generating function $P_{2k}$.  This yields the summand ${\color{blue}z C_{o,2k-1}P_{2k}}$.
\end{itemize}

Thus, the recurrence formula for our generating function $P_{2k}$ is given by 
\[
P_{2k} = 1 + {\color{red}z C_{e,2k-2} S_{2k}} + {\color{blue}z C_{o,2k-1}P_{2k}}.
\]


We find the remaining recurrence formulas for $R_{2k+1}, S_{2k}, T_{2k+1}$ in the same fashion, which results in the following system of equations:

\begin{equation} \label{eq:PRST}
\begin{cases}
\begin{aligned}
P_{2k} &= 1 + {\color{red}z C_{e,2k-2} S_{2k}} + {\color{blue}z C_{o,2k-1}P_{2k}}\\
R_{2k+1} &= \phantom{1+\ \ } {\color{blue}z C_{e,2k} T_{2k+1}} + {\color{red}z C_{o,2k-1} R_{2k+1}}\\
S_{2k} &= \phantom{1+\ \ } {\color{blue}z C_{e,2k-2} P_{2k}} + {\color{red}z R_{2k+1} S_{2k}}\\
T_{2k+1} &= 1 + {\color{red}z  P_{2k+2} R_{2k+1}} + {\color{blue}z C_{o,2k-1} T_{2k+1}}
\end{aligned}
\end{cases}
\end{equation}


Now, consider the sequence $\{|\mathfrak{D}^4_{2n}(1423)|\}_{n\ge 0}=\{|\mathfrak{D}^4_{2n}(312)|\}_{n\ge 0}$. We claim that its ordinary generating function is $R_1/z$. Indeed, removing the top row and the rightmost column of the $R$-board with $m=1$ yields exactly the board of allowed cells in a Dumont-4 permutation.

Now consider the function $R_1/z$. Solving the system of equations \eqref{eq:PRST}  for $R_{2k+1}$, $k\ge 0$, yields the following recursive formula:
\[
R_{2k+1} = \frac{zC_{e,2k}}{(1-zC_{o,2k-1})^2} \,\,
 \genfrac{}{}{0pt}{}{}{-} \,\,
\frac{z^2C_{e,2k}}{1-zC_{o,2k+1}} \,\,
 \genfrac{}{}{0pt}{}{}{-} \,\,
\frac{z^2C_{e,2k}^2}{1} \,\,
 \genfrac{}{}{0pt}{}{}{-} \,\, 
zR_{2k+3}.
\]

Multiplying both sides by $z$, we obtain for all $k\ge 0$:
\[
zR_{2k+1} = \frac{z^2C_{e,2k}}{(1-zC_{o,2k-1})^2} \,\,
 \genfrac{}{}{0pt}{}{}{-} \,\,
\frac{z^2C_{e,2k}}{1-zC_{o,2k+1}} \,\,
 \genfrac{}{}{0pt}{}{}{-} \,\,
\frac{z^2C_{e,2k}^2}{1} \,\,
 \genfrac{}{}{0pt}{}{}{-} \,\, 
zR_{2k+3}.
\]

Note, the term on the left and the last term on the right are of the same form with $k$ increasing by $1$. This yields the continued fraction representation for the generating function $R_1/z=zR_1/z^2=\sum_{n=0}^{\infty}{|\mathfrak{D}^4_{2n}(312)|z^{2n}}$. Substituting $\sqrt{z}$ for $z$, we get the generating function for $\{|\mathfrak{D}^4_{2n}(312)|\}_{n\ge 0}$.
\end{proof}

Also note that the numerators  of first three terms in the recurrence formula contain only the truncations of the even part of $C(z)$, and the denominators contain only the truncations of the odd part of $C(z)$. Moreover, all of those numerators and denominators are even functions. Furthermore, it is not difficult to see that the term $zR_{2k+3}$ in the recursive formula \eqref{eq:d4-312-ogf} does not contribute to the terms of degree at most 6 in $zR_{2k+1}$ (due to the factors of $z^2$ in the three numerators). Applying this observation iteratively, we see that deleting the term $zR_{2k+1}$ in the resulting recursive formula for $R_1/z$ does not affect the terms of degrees $2n\in[0,6k-2]$, i.e. $2n$ for $0\le n\le 3k-1$.

Finally, we note that the sequence $\{|\mathfrak{D}^4_{2n}(312)|\}_{n\ge 0}$ is \href{https://oeis.org/A343795}{A343795} in OEIS \cite{OEIS} and begins 
\[
1,1,3,10,39,174,872,4805,28474,178099,1160173,7803860,\dots .
\]
Note that these are the coefficients for $0\le n\le 11$, so $3k-1=11$ yields $k=4$, and therefore these terms can be found by expanding $R_1/z$ iteratively until we reach $zR_9$, then removing the term $zR_9$.




\section{A single occurrence of patterns in Dumont permutations}

The results of this section focus on enumeration of Dumont permutations with a single occurrence of certain patterns.

\subsection{One occurrence in Dumont-1 and Dumont-2 permutations}

We first introduce notation that is useful for going back and forth between sequences and their generating functions. Then we will review results on single occurrences of patterns in Dumont-1 and Dumont-2 permutations. Finally, we will prove a related result on a single occurrence of a pattern in Dumont-4 permutations.

\begin{notation}
For any ordinary generating function, let:
\[
A(z) \longleftrightarrow \{a_n\} \quad \text{if} \quad A(z) = \text{ogf}\{a_n\} = \sum_{n=0}^\infty a_n z^n.
\]
\end{notation}

\noindent For example, for the Catalan numbers we have $\{C_n\} \longleftrightarrow  C = C(z)$ and $C = 1 +  zC^2$.
Recall that
\[
C_n = \frac{1}{n+1} \binom{2n}{n},
\]
which implies that, for any $k \ge 1$,
\[
C^k \longleftrightarrow \frac{k}{n+k} \binom{2n+k-1}{n}.
\]


Likewise, for the central binomial coefficients $\displaystyle B_n = \binom{2n}{n}$ we have $\{B_n\} \longleftrightarrow  B = B(z)$ and $B = 1 +  2zBC.$  This implies that, for any $k \ge 1$, 
$$BC^k \longleftrightarrow \binom{2n+k}{n}.$$ 

\begin{theorem}[\cite{Mansour}] \label{thm:132;1} 
For all $n\ge 0$, there does not exist a Dumont-1 permutation containing $132$ exactly once.  That is,
\[
|\mathfrak{D}^1_{2n} (132;1)| = 0.
\]
\end{theorem}

The next few theorems list previous results \cite{Burstein_Ofodile, Ofodile} on Dumont-1 and Dumont-2 permutations with a single occurrence of certain patterns.

\begin{theorem}[\cite {Ofodile}]
For all $n\ge 0$, 
\[
|\mathfrak{D}^1_{2n} (312;1)| = 0 \quad \text{ and } \quad |\mathfrak{D}^1_{2n} (231;1)| = \binom {2n-2} {n-3}.
\]
\end{theorem}

%

Note that 
$2n-2=2(n-3)+4$, so
\[
\binom {2n-2} {n-3} = \binom {2(n-3)+4} {n-3} \longleftrightarrow z^3BC^4.
\]

\begin{theorem}[\cite {Ofodile}]
For all $n\ge 4$,
\[
|\mathfrak{D}^1_{2n} (213;1)| = C_{n-2}+\binom {2n-4} {n-4}.
\]
\end{theorem}

Similarly,
\[
C_{n-2}+\binom {2n-4} {n-4} = C_{n-2}+\binom {2(n-4)+4} {n-4} \longleftrightarrow z^{2}C+z^4BC^4.
\]

\begin{theorem}[\cite{Burstein_Ofodile}]
For all $n \ge 2$,
\[
|\mathfrak{D}^1_{2n} (321;1)| = (n-1)^2.
\]
\end{theorem}

\begin{theorem}[\cite{Burstein_Ofodile}]
For all $n \ge 2$,
\[
|\mathfrak{D}^2_{2n} (321;1)| = \frac{5}{n+3}\binom{2n}{n-2} \longleftrightarrow z^2 C^5.
\]
\end{theorem}

\begin{theorem}[\cite {Ofodile}]
For all $n\ge 2$,
\[
|\mathfrak{D}^2_{2n} (3142;1)| = \binom {2n-1} {n-2} \longleftrightarrow z^2BC^3.
\]
\end{theorem}

\begin{theorem}[\cite{Burstein_Ofodile}]
For all $n \ge 2$,
\[
|\mathfrak{D}^2_{2n} (2143;1)| = a_n b_{n+1} + b_n a_{n+1} + a_{n-1}a_n,
\]
where
\[
a_{2k} = \frac{1}{2k+1} \binom{3k}{k}, \quad a_{2k+1} = \frac{1}{k+1} \binom{3k+1}{k},
\]
and
\[
b_{2k} = \binom{3k-3}{k-2}, \quad b_{2k+1} = 2 \binom {3k-2}{k-2}.
\]
\end{theorem}

\subsection {A single occurrence of patterns in Dumont-4 permutations}

The results of this subsection focus on a single occurrence of patterns in Dumont-4 permutations.  To begin, let us consider a theorem.

\begin{theorem}[\cite {Noonan, Noonan_Zeilberger}] \label{Noo_Zeil}
\[
|S_n(123;1)| = \frac{3}{n} \binom{2n}{n-3}
\]
\end{theorem}
%
%
This result enumerates all permutations containing a single copy of pattern $321$. The original proof by both Noonan \cite {Noonan} and Noonan and Zeilberger \cite {Noonan_Zeilberger} uses a complicated induction and generating functions with multiple auxiliary parameters. A much shorter argument has been subsequently given by the first author in \cite{Burstein_Short}, which was then shortered further by Zeilberger \cite{Zeilberger_lovely} and recently generalized by B\'ona and the first author in \cite{Bona_Burstein}.  
Given the methods used in the proofs in \cite{Burstein_Short, Zeilberger_lovely}, as well as the corresponding permutation diagram of the decomposition of permutations in $S_n(321;1)$ as in \cite{Burstein_Short}, we will now prove a theorem first presented in \cite{Burstein_Jones}, enumerating a single occurrence of $321$ in Dumont-4 permutations.

\begin{theorem}[\cite{Burstein_Jones}] \label{thm:d4-321-1}
For all $n\ge 1$,
\[
\begin{aligned}
|\mathfrak{D}^4_{2n} (321;1)| &= |S_n (321;1)| + |S_{n+1} (321;1)|\\
&= \frac{3}{n}\binom{2n}{n-3}+\frac{3}{n+1}\binom{2n+2}{n-2}\\
&= C_{n+3} - 3C_{n+2} - C_{n+1} + 3C_n.
\end{aligned}
\]
\end{theorem}

It follows from Theorems \ref{Noo_Zeil} and \ref{thm:d4-321-1} that the generating function for the sequence $|\mathfrak{D}^4_{2n}(321;1)|$, $n\ge 0$, is
\[
z^3 C^6 + \frac{1}{z}\cdot z^3 C^6 = (z^2+z^3)C^6.
\]

\begin{proof}
Let $\pi \in \mathfrak{D}^4_{2n}(321;1)$. Suppose that the single occurrence of pattern $321$ is formed by values $c>b>a$ and $\pi=\tau_1c\tau_2b\tau_3a\tau_4$ for some strings $\tau_i$, $i=1,2,3,4$. Then $b$ is a fixed point. Let $\pi'$ and $\pi''$ be the patterns of $\tau_1c\tau_2a$ and $c\tau_3a\tau_4$, respectively, as in \cite{Zeilberger_lovely}.  Then $\pi'$ and $\pi''$ avoid $321$, $\pi'$ does not end on its top entry, and $\pi''$ does not start with its bottom entry. Moreover, the entry $a$, the ``1'' of the occurrence of $321$, is a deficiency and hence an even value in an even position in $\pi$. 

Consider two cases based on the parity of $b$, the entry ``$2$'' of the single $321$-occurrence in $\pi$.

\begin{enumerate}[\emph{Case} 1:]

\item \label{enum:b-even} Suppose $b$ is even, say, $b=2k$ for some $k\in[2,n-1]$ (note that the entry $b$ cannot be either $2$ or $2n$, since neither is involved as a ``2'' in an occurrence of $321$: the former because $1$ is a fixed point, the latter because $2n$ is the largest entry).

Then $|\pi'|=2k$ and $|\pi''|=2n-2k+1$. Recall that $a$ is even; moreover, it occupies the last position in $\pi'$, i.e. position $2k$, which is even. Therefore, $\pi'$ is also Dumont-4, and thus $\pi'\in\mathfrak{D}^4_{2k}(321)$. For what follows after this case analysis, let $\rho'=\pi'$.

Now consider the permutation $\rho''=(1,\pi''+1)$.  Note that the entry 2 in $\rho''$ is not a fixed point, since $\pi''$ does not begin with its lowest entry, thus 2 is a deficiency of $\rho''$. Therefore, $\rho''$ is also Dumont-4, and thus $\rho''\in\mathfrak{D}^4_{2n-2k+2}(321)$.

Note that the sizes of $\rho'$ and $\rho''$ add up to $2k+(2n-2k+2)=2n+2$ in this case.

\item \label{enum:b-odd} Suppose $b$ is odd, say, $b=2k+1$ for some $k\in[1,n-1]$.  Then $|\pi'|=2k+1$ and $|\pi''|=2n-2k$.  

Insert the new entry $2k+2$ in the second rightmost position of $\pi'$ to form a permutation $\rho'$.  Then $|\rho'|=2k+2$, which is even. Moreover, as in Case \ref{enum:b-even}, the entry $a$, which is at most $2k$ and even (as a deficiency), now occupies position $2k+2$, also even. Therefore, $\rho'\in\mathfrak{D}^4_{2k+2}(321)$.

Now form the permutation $\rho''$ from $\pi''$ in the following fashion: if $\pi''=(\pi_1, 1, \pi_2)$, then $\rho''=(1,3,\pi_1+3,2,\pi_2+3)$. Note that $|\rho''|=|\pi''|+2=2n-2k+2$ all deficiencies of $\rho''$ come from deficiencies of $\pi$. 

The entry $a$ in $\pi$ becomes the entry $2$ of $\rho''$, and hence also an even deficiency. Its position in $\rho''$ is $(\rho'')^{-1}(2)=\pi^{-1}(a)-(2k+1)+1+2=\pi^{-1}(a)-(2k-2)$ is also even, since $\pi^{-1}(a)$ is even. The entries of $\rho''$ that are at least $4$ are obtained by subtracting $2k-2$ from the corresponding entries of $\pi$. However, their position also shifted to the left by $2k-2$, so the parity of their positions and values as well as the placement relative to the diagonal (i.e. being a fixed point, excedance, or deficiency) remains the same. Therefore, $\rho''$ is a Dumont-4 permutation, and thus $\rho''\in\mathfrak{D}_{2n-2k+2}(321)$.

Note that the sizes of $\rho'$ and $\rho''$ add up to $(2k+2)+(2n-2k+2)=2n+4$ in this case.

\end{enumerate}


Now, given permutations $\rho'$ and $\rho''$, use the map of Theorem \ref{thm:1432} to produce two permutations $\sigma'$ and $\sigma''$, respectively, that avoid $321$. Note that $\rho'$ does not end on its top entry, and therefore neither does $\sigma'$. Likewise, $\rho''$ does not start with $12\dots$, and therefore $\sigma''$ does not start with $1$. Thus, the pair $\sigma'$ and $\sigma''$ are as in the proof in \cite {Burstein_Short} with the combined length of $(2n+2)/2=n+1$ or $(2n+4)/2=n+2$. Thus, we can combine them as in \cite {Burstein_Short} to produce a single permutation $\sigma$ of length $n$ or $n+1$ (depending on the parity of $b$) with a single occurrence of pattern $321$.

Therefore, the number of Dumont-4 permutations of length $2n$ that have a single 321-occurrence is equal to the number of permutations of length $n$ with a single 321-occurrence plus the number of permutations of length $n+1$ with a single 321-occurrence.  That is, $$|\mathfrak{D}^4_{2n} (321;1)| = |S_n (321;1)| + |S_{n+1} (321;1)|.$$  Finally, Zeilberger \cite {Zeilberger_lovely} showed that $|S_n (321;1)| = C_{n+2} - 4C_{n+1} + 3C_n$.  It follows that 
\[
|\mathfrak{D}^4_{2n} (321;1)| =( C_{n+2} - 4C_{n+1} + 3C_n)+ ( C_{n+3} - 4C_{n+2} + 3C_{n+1})
= C_{n+3} - 3C_{n+2} - C_{n+1} + 3C_n. \qedhere
\]
\end{proof}

\begin{example} 
Let $\pi = 135462 \in \mathfrak{D}^4_{6}(321;1)$.  Here, we see that the occurrence of $321$ is the subsequence $542$, where the ``$2$" is the entry $4$, an even fixed point. The map of the proof in \cite{Burstein_Short} yields $\pi'=1342$ and $\pi''=231$, so that $\rho'=1342$ and $\rho''=1342$ as well. Now applying the map of Theorem \ref{thm:1432} yields $\sigma'=21$ and $\sigma''=21$. Combining those again as in \cite{Burstein_Short} yields $\sigma=321$. Note that $|\sigma|=|\pi|/2$.

Similarly, let $\pi = 136254 \in \mathfrak{D}^4_{6}(321;1)$.  Here, we see that the occurrence of $321$ is the subsequence $654$, where the ``$2$" is the entry $5$, an odd fixed point. The map of the proof in \cite{Burstein_Short} yields $\pi'=13524$ and $\pi''=21$, so that $\rho'=135624$ and $\rho''=1342$. Now applying the map of Theorem \ref{thm:1432} yields $\sigma'=312$ and $\sigma''=21$. Combining those again as in \cite{Burstein_Short} yields $\sigma=4132$. Note that $|\sigma|=|\pi|/2+1$.
\end{example}




\section{Conjecture and conclusion}

\subsection {A conjecture on restricted Dumont-1 permutations}
\label {B-J Conjecture}

The table at the end of \cite[Section 3.1]{Burstein_Restricted} implies that no two Dumont-1 permutations of length $4$ (recall: $\mathfrak{D}^1_4 = \{2143, 3421, 4213\}$) are Wilf-equivalent.  However, we found that this is due to an incorrect calculation of $|\mathfrak{D}_{10}(3421)|$, which is, in fact, equal to $|\mathfrak{D}_{10}(2143)|$. Furthermore, we found that $|\mathfrak{D}_{10}(3421)|=|\mathfrak{D}_{10}(2143)|$ for $n \le 6$; an additional computation by Albert \cite{Albert_comm} showed that the two sequences are the same for $n\le 10$, i.e. for Dumont-1 permutations of length up to $20$.  Therefore, we conjecture that the two sequences are equinumerous for all $n$.

\begin{conjecture}[\cite{Burstein_Jones}]
The following Wilf-equivalence holds on Dumont-1 permutations:
\[
2143 \sim 3421 \quad \text{on} \quad \mathfrak{D}^1_{2n}, \quad \text{that is} \quad |\mathfrak{D}^1_{2n} (2143)| = |\mathfrak{D}^1_{2n} (3421)| \quad \text{for all $n\ge 0$}.
\]
\end{conjecture}

Our conjecture is based on the following empirical observation. In all cases where sequences $|S_n(\pi_1)|$ and $|S_n(\pi_2)|$ are known to diverge (where $\pi_1$ and $\pi_2$ are patterns of the same length $k$), the first index at which the corresponding terms differ was found to be at most $2k-1$. Similarly, we expect Dumont permutations avoiding patterns of length $k$ to be Wilf-equivalent if their avoidance sequences coincide for avoiding permutations of length up to $4k-2$. In our case, the patterns are of length $k=4$, and the sequences coincide up to $n=10$, i.e. up to length $20$, greater than $4k-2=14$. Thus, we conjecture that they coincide for all $n\ge 0$.  See Table \ref{table:D1} for the sequence of the enumeration of $|\mathfrak{D}^1_{2n} (2143)| = |\mathfrak{D}^1_{2n} (3421)|$ up to $n = 10$.

\begin{table}[ht]
\begin{center}
\begin{tabular}{c||c|c|c|c|c|c|c|c|c|c|c}
	$n$ &0 & 1 & 2 & 3 & 4 & 5 & 6 & 7 & 8 &9 &10\\
	\hline 
	$|\mathfrak{D}^1_{2n} (2143)|$ & 1 & 1 & 2 & 7 & 36 & 239 & 1,892 & 17,015 & 168,503 & 1,799,272 & 20,409,644\\
	\hline 
	$|\mathfrak{D}^1_{2n} (3421)|$ & 1 & 1 & 2 & 7 & 36 & 239 & 1,892 & 17,015 & 168,503 & 1,799,272 & 20,409,644
\end{tabular}
\caption{Sequence of the cardinalities $|\mathfrak{D}^1_{2n} (2143)| = |\mathfrak{D}^1_{2n} (3421)|$}
\label{table:D1}
\end{center}
\end{table}


This conjecture has turned out to be very difficult to prove, and only some partial inroads are made as of this writing. One approach to proving such a conjecture is to refine it by considering distributions of some combinatorial statistics on both $\mathfrak{D}^1_{2n} (2143)$ and $\mathfrak{D}^1_{2n} (3421)$. For this, we will first need to define a generalized version of a permutation pattern.

\begin{definition}[\cite{Babson}]
A \emph{vincular} (or \emph{generalized}, or \emph{dashed}) permutation pattern is a pattern that allows the additional requirement that two (or more) adjacent letters in a pattern be also adjacent in the containing permutation.
\end{definition}

Note, in the pattern $2\text{-}31$, the ``$2$'' and ``$3$'' need not be adjacent, however the ``$3$'' and the ``$1$'' must be adjacent.  Now that we have the definition for vincular patterns, we can refine our conjecture slightly.  After studying the behavior of distributions of occurrences of various vincular patterns on $\mathfrak{D}^1_{2n}(2143)$ and $\mathfrak{D}^1_{2n}(3421)$, we were able to form the following conjecture.

\begin{conjecture}[\cite{Burstein_Jones}]
For all $n \ge 0$, we conjecture the following.
Let
\[
a_{n,k} = |\{\pi \in \mathfrak{D}^1_{2n} (2143) | (2\text{-}31)\pi = k\}|, \quad
b_{n,k} = |\{\pi \in \mathfrak{D}^1_{2n} (3421) | (13\text{-}2)\pi = k\}|,
\]
where $(2\text{-}31)\pi$ (resp. $(13\text{-}2)\pi$) is the number of occurrences of $2\text{-}31$ (resp. $13\text{-}2$) in $\pi$. Then $a_{n,k} = b_{n,k}=1$ for $k=\binom{n}{2}$, $a_{n,k} = b_{n,k}=0$ for $k>\binom{n}{2}$, and
\[
\sum_{k=0}^{m}a_{n,k} \ge \sum_{k=0}^{m} b_{n,k} \quad \text{ for } 0 \le m \le \binom{n}{2},
\]
with equality for $m = \binom{n}{2}$.
\end{conjecture}

This is less optimal than statistic distributions that are equal to each other; however, of the statistics we studied, this relation is the only interesting one we found. Additionally, we conjecture that both sequences $(a_k)$ and $(b_k)$ are unimodal for each $n\ge 0$, equal at the tails of the distributions, while in the middle of the distributions we first have a block of $a_k > b_k$ followed by a block of $a_k < b_k$ with the switch in direction of the inequality occurring at approximately $k=2n-5$.  See Table \vref{table:cum_dist1} for $n=5$, Table \vref{table:cum_dist2} for $n=6$, and Table \vref{table:cum_dist3} for $n=7$.


\begin{table}[!h]
\begin{center}
\begin{tabular}{c||c|c|c|c|c|c|c|c|c|c|c||c}
	$k$  &0 & 1 & 2 & 3 & 4 & 5 & 6 & 7 & 8 &9 &10 & \textbf{Total}\\
	\hline 
	$a_{5,k}$ & 1 & 10 & 30 & 45 & 49 & 42 & 31 & 18 & 9 & 3 & 1 & \textbf{239}\\
	\hline
	$b_{5,k}$ & 1 & 10 & 29 & 44 & 48 & 43 & 32 & 19 & 9 & 3 & 1 & \textbf{239}\\
	\hline
	$a_{5,k} \, \openbox \, b_{5,k}$ & $=$ & $=$ & $>$ & $>$ & $>$ & $<$ & $<$ & $<$ & $=$ & $=$ & $=$
\end{tabular}
\caption{$|\{\pi \in \mathfrak{D}^1_{10} (2143) | (2\text{-}31)\pi = k\}|$ versus $|\{\pi \in \mathfrak{D}^1_{10} (3421) | (13\text{-}2)\pi = k\}|$}
\label{table:cum_dist1}
\end{center}
\end{table}


\begin{table}[!h]
\begin{flushleft}
\begin{tabular}{c||c|c|c|c|c|c|c}
	$k$  &0 & 1 & 2 & 3 & 4 & 5 & 6 \\
	\hline 
	$a_{6,k}$ & 1 & 15 & 70 & 160 & 246 & 298 & 303\\
	\hline
	$b_{6,k}$ & 1 & 15 & 65 & 147 & 228 & 284 & 302 \\
	\hline
	$a_{6,k} \, \openbox \, b_{6,k}$ & $=$ & $=$ & $>$ & $>$ & $>$ & $>$ & $>$
\end{tabular}
\end{flushleft}
%
\begin{flushright}
\begin{tabular}{c||c|c|c|c|c|c|c|c|c||c}
	$k$ &7 & 8 &9 &10 &11 &12 &13 &14 &15 & \textbf{Total}\\
	\hline 
	$a_{6,k}$ &268 & 208 & 145 & 89 & 49 & 24 & 11 & 4 & 1 & \textbf{1892}\\
	\hline
	$b_{6,k}$ &277 & 223 & 157 & 98 & 53 & 26 & 11 & 4 & 1 & \textbf{1892}\\
	\hline
	$a_{6,k} \, \openbox \, b_{6,k}$ & $<$ & $<$ & $<$ & $<$ & $<$ & $<$ & $=$ & $=$ & $=$
\end{tabular}
\end{flushright}
\caption{$|\{\pi \in \mathfrak{D}^1_{12} (2143) | (2\text{-}31)\pi = k\}|$ versus $|\{\pi \in \mathfrak{D}^1_{12} (3421) | (13\text{-}2)\pi = k\}|$}
\label{table:cum_dist2}
\end{table}

\begin{table}[!h]
\begin{flushleft}
\begin{tabular}{c||c|c|c|c|c|c|c|c|c}
	$k$  &0 & 1 & 2 & 3 & 4 & 5 & 6 & 7 & 8\\
	\hline 
	$a_{7,k}$ & 1 & 21 & 140 & 455 & 945 & 1497 & 1956 & 2215 & 2226\\
	\hline
	$b_{7,k}$ & 1 & 21 & 125 & 388 & 804 & 1294 & 1760 & 2089 & 2211 \\
	\hline
	$a_{7,k} \, \openbox \, b_{7,k}$ & $=$ & $=$ & $>$ & $>$ & $>$ & $>$ & $>$ & $>$ & $>$ 
\end{tabular}
\end{flushleft}
%
\begin{flushright}
\begin{tabular}{c||c|c|c|c|c|c|c|c|c|c|c|c|c||c}
	$k$  &9 &10 &11 &12 &13 &14 &15 &16 &17 &18 &19 &20 &21 & \textbf{Total}\\
	\hline 
	$a_{7,k}$ & 2032 & 1700 & 1317 & 948 & 641 & 410 & 249 & 140 & 72 & 32 & 13 & 4 & 1 & \textbf{17015}\\
	\hline
	$b_{7,k}$ & 2111 & 1840 & 1472 & 1091 & 750 & 482 & 288 & 158 & 78 & 34 & 13 & 4 & 1 & \textbf{17015}\\
	\hline
	$a_{7,k} \, \openbox \, b_{7,k}$ & $<$ & $<$ & $<$ & $<$ & $<$ & $<$ & $<$ & $<$ & $<$ & $<$ & $=$ & $=$ & $=$
\end{tabular}
\end{flushright}
\caption{$|\{\pi \in \mathfrak{D}^1_{14} (2143) | (2\text{-}31)\pi = k\}|$ versus $|\{\pi \in \mathfrak{D}^1_{14} (3421) | (13\text{-}2)\pi = k\}|$}
\label{table:cum_dist3}
\end{table}

\subsection{Conclusion}

In this paper, we have enumerated Dumont-4 permutations avoiding certain patterns, as well as Dumont-4 permutations containing a single occurrence of certain patterns. 
The Catalan numbers and powers of $2$ that occur in the enumeration of permutations avoiding three-letter patterns were also encountered in our results. However, for some patterns we could only find a generating function in the form of a continued fraction, as in \cite{RWZ}.

We also gave an intriguing conjecture regarding the Wilf-equivalence of a pair of patterns on Dumont permutations of the first kind. This conjecture has proven quite unyielding, and despite being presented twice at the open problem sessions at Permutation Patterns 2016 and 2018, progress on it is yet to be made. 

It would also be interesting to see if other Wilf-equivalences exist on Dumont permutations of the first kind avoiding a single four-letter pattern. The same question may be posed for the other kinds of Dumont permutations as well.


\subsection*{Acknowledgments}

The authors would like to thank Michael Albert and Lindsey-Kay Lauderdale for computing several terms of sequences $|\mathfrak{D}^1_{2n} (2143)|$ and $|\mathfrak{D}^4_{2n} (312)|$, as well as the anonymous referee for a thorough reading of the paper and many helpful comments and suggestions.



{\small

\begin{thebibliography}{99}

\bibitem{Albert_comm} M. Albert, personal communication, 2016.


\bibitem{Babson} E. Babson, E. Steingr\'imsson, Generalized permutation patterns and a classification of the Mahonian statistics, \emph{S\'em. Lothar. Combin.}, \textbf{44} (2000), Research Article B44b, 18 pp.


\bibitem{Bevan} D. Bevan, \emph{Permutation patterns: basic definitions and notation}, \href{https://arxiv.org/abs/1506.06673}{arxiv:1506.06673}.

\bibitem{Bona_Burstein} M. B\'ona, A. Burstein, Permutations with exactly one copy of a decreasing pattern of length $k$, \emph{Ann. Combin.}, to appear, preprint at \href{https://arxiv.org/abs/2101.00332}{arxiv:2101.00332}.

\bibitem{Bona} M. B\'ona, \emph{Combinatorics of Permutations}. Chapman \& Hall/CRC, Boca Raton, FL, 2004.



\bibitem{Burstein_Short} A. Burstein, A Short Proof for the Number of Permutations Containing Pattern 321 Exactly Once, \emph{Electron. J. Combin.} \textbf{18} (2011).



\bibitem{Burstein_Restricted} A. Burstein, Restricted Dumont permutations, \emph{Ann. Combin.} \textbf{9} (2005), 269-280.

\bibitem{Burstein_Dyck} A. Burstein, S. Elizalde, T. Mansour, Restricted Dumont permutations, Dyck paths and noncrossing partitions, \emph{Discrete Math.} \textbf{306} (2006), no. 22, 2851-2869. Extended abstract in \emph{Proceedings of FPSAC 2006}, June 2006, San Diego, CA.

\bibitem{Burstein_Jones} A. Burstein, O. Jones, Restricted Dumont Permutations, talk, Special Session on Enumerative Combinatorics, AMS Central Section Meeting, University of St. Thomas, Minneapolis, MN, October 2016.

\bibitem{Burstein_Ofodile} A. Burstein, C. Ofodile, Dumont permutations containing one occurrence of certain three and four letter patterns, talk, Permutation Patterns 2011, California Polytechnic University, San Luis Obispo, CA, June 2011.

\bibitem {Burstein_New} A. Burstein, M. Josuat-Verg\`es, W. Stromquist, New Dumont permutations, \emph{Pure Math. Appl. (Pu.M.A.)} \textbf{21} (2010), no. 2, 177-206.

\bibitem{Burstein_Third} A. Burstein, W. Stromquist, Dumont permutations of the third kind, extended abstract, in \emph{Proceedings of FPSAC 2007}, July 2007, Tianjin, China.

\bibitem {Dumont} D. Dumont, Interpr\'etations combinatoires des nombres de Genocchi, \emph{Duke J. Math.} \textbf {41} (1974), 305-318.

\bibitem {Dumont_Foata} D. Dumont, D. Foata, Une propri\'et\'e de sym\'etrie des nombres de Genocchi, \emph{Bul. Soc. Math. France} \textbf{104} (1976), no. 4, 433-451.

\bibitem{Foata} Dominique Foata, On the Netto inversion number of a sequence, \emph{Proc. Amer. Math. Soc.} \textbf {19} (1968), 236-240.


\bibitem {Jones} O. Jones, \emph{Enumerations of Dumont permutations avoiding certain four-letter patterns}, Ph.D. thesis, Howard University, 2019.

\bibitem {Kitaev_Remmel_1} S. Kitaev, J. Remmel, Classifying descents according to equivalence mod $k$, \emph{Electron. J. Combin.} \textbf{13(1)} (2006), \#R64.

\bibitem {Kitaev_Remmel_2} S. Kitaev, J. Remmel, Classifying descents according to parity, \emph{Ann. Combin.} \textbf{11} (2007), 173-193.

\bibitem{Krattenthaler} C. Krattenthaler, Permutations with restricted patterns and Dyck paths, \emph{Adv. Appl. Math.} \textbf{27} (2001), 510-530.

\bibitem {Knuth} D. E. Knuth, \emph{The Art of Computer Programming}, vol. 3. Addison-Wesley Publishing, MA 1973.


\bibitem {Mansour} T. Mansour, Restricted 132-Dumont permutations, \emph{Australasian J. Combin.} \textbf{29} (2004), 103-117.



\bibitem {Noonan} J. Noonan, The number of permutations containing exactly one increasing subsequence of length three, \emph{Discrete Math.} \textbf{152} (1996), no. 1-3, 307-313.

\bibitem{OEIS} OEIS Foundation Inc. (2021), \emph{The On-Line Encyclopedia of Integer Sequences}, \href{http://oeis.org}{http://oeis.org}.

\bibitem {Ofodile} C. Ofodile, \emph{The enumeration of Dumont permutations with few occurrences of three- and four-letter patterns}, Ph.D. thesis, Howard University, 2011.

\bibitem {Noonan_Zeilberger} J. Noonan, D. Zeilberger, The enumeration of permutations with a prescribed number of ``forbidden" patterns, \emph{Adv. Appl. Math.} \textbf{17} (1996), no. 4, 381-407.

\bibitem {Simion} R. Simion, F. W. Schmidt, Restricted permutations,  \emph{Europ. J. Combin.} \textbf {6} (1985), 383-406.



\bibitem{RWZ} A. Robertson, H.S. Wilf, D. Zeilberger, Permutation patterns and continued fractions, \emph{Electron. J. Combin.} \textbf{6} (1999), \#R38, 6 pp.

\bibitem{Vatter} V. Vatter, ``Permutation Classes'', in \emph{Handbook of Enumerative Combinatorics}, M. B\'ona, ed., Chapman and Hall/CRC, 2015. Available online at \href{https://arxiv.org/abs/1409.5159}{arxiv:1409.5159}.


\bibitem {Zeilberger_lovely} D. Zeilberger, Alexander Burstein's lovely combinatorial proof of John Noonan's beautiful theorem that the number of $n$-permutations that contain the pattern $321$ exactly once equals $(3/n)(2n)!/((n-3)!(n+3)!)$, \emph{Pure Math. Appl.} \textbf{22} (2011), no. 2, 297-298.

\end{thebibliography}

}
\end{document}